\documentclass{amsart}
 \usepackage{amsaddr}
 
\usepackage{graphicx, nicefrac}

\usepackage[utf8]{inputenc}
\usepackage[T1]{fontenc}
\usepackage[a4paper,left=2.5cm,right=2.5cm,top=2.5cm,bottom=2.5cm]{geometry}
\usepackage{libertine}
\usepackage{graphicx}
\usepackage{yfonts}
\usepackage{float}
\usepackage{amsmath}
\usepackage{amsfonts}
\usepackage{multicol}
\usepackage{array}
\usepackage{multirow}
\usepackage{listings}
\usepackage{appendix}
\usepackage{subfigure}
\usepackage{mathtools}
\usepackage{bigints}
\usepackage{caption}
\usepackage{amsthm}
\usepackage{bbold}
\usepackage{verbatim}
\usepackage[shortlabels]{enumitem}
\usepackage{tikz-cd}

\makeatletter

\newcommand{\vertiii}[1]{{\left\vert\kern-0.25ex\left\vert\kern-0.25ex\left\vert #1 
    \right\vert\kern-0.25ex\right\vert\kern-0.25ex\right\vert}}
\newcommand{\stars}{}
\DeclareRobustCommand{\stars}[1]{\stars@{#1}}
\newcommand{\stars@}[1]{%
  \ifcase#1\relax\or\stars@one\or\stars@two\or\stars@three\or\stars@four
  \else ??\fi
}

\usepackage{amsmath}
\usepackage{cancel}
\usepackage{amssymb}
\usepackage[dvips]{epsfig}
\usepackage{graphicx}
\usepackage{latexsym}
\usepackage{amsthm}
\usepackage{amssymb}
\usepackage{ulem}

\usepackage{amsmath,amsthm,amsfonts,amssymb,amscd,amsbsy,dsfont,hyperref}

\usepackage{palatino}

\usepackage{eucal}
\usepackage{float}
\usepackage{xcolor}

\usepackage{xurl}

\usepackage{color}

\usepackage{multirow}

\newcommand*\dif{\mathop{}\!\mathrm{d}}

\newcommand{\Prob}{\mathbb{P}}

\newcommand{\R}{\mathbb{R}}

\DeclareMathOperator*{\argmax}{arg\,max}
\DeclareMathOperator*{\argmin}{arg\,min}

\newtheorem{remark}{Remark}
\newtheorem{corollary}{Corollary}
\newtheorem{definition}{Definition}

\newtheorem{proposition}{Proposition}
\newtheorem{theorem}{Theorem}

\newcommand{\C}{\mathcal{C}}

\newcommand{\B}{\mathbb{B}}

\newtheorem{thm}{Theorem}
\newtheorem{Proposition}{Proposition}
\newtheorem{Lemma}{Lemma}
\newtheorem{assumption}[thm]{Assumption}

\newtheorem*{assumption*}{Assumption}

\begin{document}

\title{A non-local estimator for locally stationary Hawkes processes}
\author[]{\footnotesize Thomas Deschatre$^1$ \hspace{0.2cm} Pierre Gruet$^1$ \hspace{0.2cm} Antoine Lotz$^{1,2}$ \vspace{1em}}
\address{\footnotesize	$^1$\textsc{EDF} \textsc{R}\&\textsc{D} , \textsc{F}i{me}  , Palaiseau, France
\\\vspace{-0.1em}
$^1$ \textsc{Université Paris Dauphine - PSL}, \textsc{PSL}, Paris, France}
\begin{abstract}
      We consider the problem of estimating the parameters of a non-stationary Hawkes process with time-dependent reproduction rate and baseline intensity. Our approach relies on the standard maximum likelihood estimator (\textsc{mle}), coinciding with the conventional approach for stationary point processes characterised by~\cite{OgataMLE}. In the fully parametric setting, we find that the \textsc{mle} over a single observation of the process over $[0,T]$ remains consistent and asymptotically normal as $T \to \infty$. Our results extend partially to the semi-nonparametric setting where no specific shape is assumed for the reproduction rate $g \colon [0,1] \mapsto \R_+$. We construct a time invariance test with null hypothesis that $g$ is constant against the alternative that it is not, and find that it remains consistent over the whole space of continuous functions of $[0,1]$. As an application, we employ our procedure in the context of the German intraday power market, where we provide evidence of fluctuations in the endogeneity rate of the order flow. \\
      
   \noindent \textbf{Mathematics Subject Classification (2020)}: \textsc{62F03, 62F12, 60G55}\\
\noindent \textbf{Keywords}: Hawkes processes, locally stationary process, maximum likelihood estimation. 
\\
   
\end{abstract}

\maketitle

\begin{center}
    
\end{center}

\section{Introduction}\label{section:introduction}
\subsection{Context and motivation} A Hawkes process $(\boldsymbol{N}_t)=(N_{k,t})$, $k =1 \hdots K$, is a path-dependent point process, wherein past events in a component $(N_{k,t})$ may induce later occurrences at any other $(N_{l,t})$ via mutual excitation. The structure of the process is encoded within a kernel function $\phi=(\phi_{kl}) \colon \R \mapsto \R_+^{K} \times \R_+^{K}$, with the integral $\int_0^{\infty} \phi_{kl}(s) \dif s$ quantifying the strength of the influence of $(N_{l,t})$ upon $(N_{k,t})$. Formally, the intensity $(\lambda_{k,t})$ of the process writes
\begin{equation*}
    \lambda_{k,t} = \mu_k + \sum_{l=1}^K\int_0^t \phi_{kl}(t-s) \dif N_{l,s},
\end{equation*}
where $\mu_k \in \R_+$ and the integral runs over $[0,t)$. In its seminal construction at $K=1$ due to~\cite{Oakes}, the process is endowed a cluster representation. Clusters arrive following a Poisson process with rate $\mu$, and expand according to a Galton-Watson dynamic with mean offspring number $\varrho=\int_0^{\infty} \phi(s) \dif s$. The parameter $\varrho$ is thus referred to as a reproduction rate, or endogeneity ratio. The process $(\boldsymbol{N}_t)$ is then known to reach stationarity in the sub-critical case $\varrho<1$. The cluster representation and the stationarity condition both extend to the multivariate setting (see e.g~\cite{Reynaud} and \cite{bremaud1996stability}), the latter then depending on the spectral radius of $\int_0^{\infty} \phi(s) \dif s$.\\

Hawkes processes have found productive use cases in seismology \cite{OgataETAS}, neuroscience \cite{Reynaud}, statistical finance \cite{bacrymicrostructure}, among other domains. In the wake of these many applications, some recent attention has been dedicated to time-dependent extensions of the model. Time-dependent cluster arrivals were introduced in \cite{FengInferenceForNonStationarySEPP} so as to model non-stationary phenomena in financial markets, and time-dependent reproduction kernels $\phi$ considered by~\cite{ROUEFF1}, with comparable applications. The process $(\boldsymbol{N}^T_t)$ then lives on the compact timeframe $[0,T]$, its kernel depending on the normalized duration $\nicefrac{t}{T}$ and its intensity generally expressing as
\begin{equation}\label{equ:general_intensity}
    \lambda^T_{k,t}
    =
    \Phi_k
    \Bigl[ 
    \mu_k\bigl( \frac{t}{T}\bigr)
    +
    \sum_{l=1}^K
    \int_0^t \phi_{l,s}\bigl(t-s,\frac{t}{T} \bigr)
    \dif N^T_{l,s}
    \Bigr], \hspace{0.1cm} t \in [0,T]
\end{equation}
where $k=1\hdots K$, the $\Phi_k \colon \R \mapsto \R_+$ are activation functions, and $\mu_k \colon [0,1] \mapsto \R_+$ The results of~\cite{ROUEFF1} include the formal basis for the estimation of the model. This is completed in~\cite{roueff2}, where the asymptotics of local mean density and local Bartlett estimators are characterized. An analogous line of research is followed in \cite{Clinet}, where the convergence of local maximum likelihood estimators (local \textsc{mle}s) is obtained in the case of gamma kernels $\phi(s,\nicefrac{t}{T})  =  g(\nicefrac{t}{T}) s^m \exp(-\beta(\nicefrac{t}{T}) s)$ and mixtures thereof, where $g$ and $\beta$ are positive-valued real functions and $m \in \mathbb{N}^\star$. Finally, \cite{MammenMuller} obtain the convergence of local polynomial estimators in the compact support case $\phi(s,\nicefrac{t}{T}) = g(\nicefrac{t}{T}) \mathbb{1}_A(s)$, where $A$ is a known compact interval of $\R_+$. \\

In a companion paper~\cite{deschatre2025limittheoremslocallystationary}, we consider a subclass of locally stationary Hawkes processes~\eqref{equ:general_intensity} with linear activation $\Phi_k(x) \propto x$ and separable kernel
  \begin{equation}\label{equ:separable_kernel}
      \phi_{kl}\bigl(s,\frac{t}{T} \bigr) = g\bigl(\frac{t}{T} \bigr) \varphi_{kl}(s),
  \end{equation}
  where $g \colon [0,1] \mapsto \R_+$ and $\varphi_{kl} \colon \R_+ \mapsto \R_+$. The function $g$ thus corresponds to a time-dependent reproduction rate in the representation of~\cite{Oakes}. We find that the asymptotic behaviour of $(\boldsymbol{N}^T_t)$ may then be assimilated to that of its stationary counterpart, in that it obeys a functional central limit Theorem generalizing the limit results found in the case $\mu \equiv 1$ and $g \equiv 1$ (see \cite{bacrylimit}).  In this article, we show that a similar correspondence holds regarding the estimation of the process. Working in the fully parametric setting, we prove that the conventional maximum likelihood estimator (\textsc{mle}) for stationary point processes, as introduced in~\cite{OgataMLE}, may be indifferently applied to non stationary processes with intensity~\eqref{equ:general_intensity} and kernel~\eqref{equ:separable_kernel}, for any Lipschitz-continuous $(\Phi_k)$ and any sufficiently integrable $(\varphi_{kl})$. We also deduce a semi-nonparametric test for time-dependencies in the reproduction rate $g$.  In contrast with the results of \cite{roueff2}\cite{MammenMuller}, and \cite{Clinet}, the resulting procedure is not specifically adapted to the locally stationary context. Hence we refer to the present approach as \textit{naïve} one.\\
  
  We are motivated by~\cite[Chapter 3]{kwan2023asymptotic}, where the consistency and weak convergence of the conventional \textsc{mle} is obtained in the univariate case $K =1$ for a time-dependent baseline $\mu$, a constant reproduction rate $g \equiv 1$, a linear activation $\Phi_k(x) \propto x$ and any sufficiently integrable $\varphi$. We find that their results may be extended to any intensity of the form~\eqref{equ:separable_kernel} as detailed in the following subsection.

\subsection{Contributions} Suppose that one observes over $[0,T]$ a single trajectory of a locally stationary Hawkes process $(\boldsymbol{N}^T_t)$ with intensity
\begin{equation*}
    \lambda^T_{k,t}(\eta,g) =  \Phi_k\Bigl[ \mu\big( \frac{t}{T},\eta \big) + \sum_{l=1}^K \int_0^t  g\big( \frac{t}{T}\big)\varphi_{kl}(t-s,\eta) \dif N^T_s \Bigr],
\end{equation*}
where $\eta$ lies in a compact subspace of $\R^p$ with non-empty interior, and $g \in C[0,1]$. Our primary concern lies with the time-inhomogeneous kernel $k^T(t,s,\eta,g) = g(\frac{t}{T}) \varphi(t-s,\eta)$, and the amplitude of its fluctuations as $\nicefrac{t}{T}$ varies. In the fully parametric setting where one has access to a family $g(\cdot,\varpi)$ indexed on $\varpi \in \R^{d+1}$ containing the true reproduction rate, we obtain in Theorem~\ref{thm:TCL} a central limit Theorem for the maximum likelihood estimator (\textsc{mle}) of $\vartheta = (\eta,\varpi)$, generalizing some prior results of Ogata~\cite{OgataMLE} and~\cite{kwan2023asymptotic} to~\eqref{equ:general_intensity}. In the more general case where no specification is imposed upon $g$ apart from continuity, we construct a test with null hypothesis that $g$ is constant, against the general alternative that it is not.  Our test leverages the results of the parametric sub-case, and relies on a maximum likelihood estimator $\hat{g}^\mathfrak{d}_T$ of $g$ defined as a combination of Bernstein polynomials of fixed degree $\mathfrak{d} \in \mathbb{N}^\star$
\begin{equation*}
    \hat{g}^\mathfrak{d}_T\bigl(\frac{t}{T} \bigr)
    =
    \sum_{k=0}^{\mathfrak{d}} \hat{\varpi}_T \binom{\mathfrak{d}}{k} \bigl(1-\frac{t}{T}\bigr)^{\mathfrak{d}-k} \bigl(\frac{t}{T}\bigr)^k,
\end{equation*}
where $(\hat{\varpi}_i)$ maximizes the model's likelihood.  We characterize the asymptotic distribution of the associated likelihood ratio statistic $\Lambda^\mathfrak{d}_T$ as defined in~\ref{def:LRS}. Under the null hypothesis, the problem reduces to a purely parametric question, and we recover in Theorem~\ref{tm:LRT_for_g_constant_or_not} the classical convergence of $\Lambda^\mathfrak{d}_T$ towards a $\chi^2$ law with $\mathfrak{d}$ degrees of freedom.  Under the alternative, we prove in Proposition~\ref{Proposition:robust} that, provided $\mathfrak{d}$ is above a certain explicit threshold depending only on $g$ but not on $T$, the method provides a consistent test for any non-constant $g \in C[0,1]$. While we allow for a non-linear intensity, meaning that $\nicefrac{\Phi_k(x)}{x}$ is not necessarily constant, our setting forbids inhibitive interactions, in the sense that the reproduction rate $g$ and the kernels $\varphi_{kl}$ should take non-negative values. This typically results in positivity constraints on certain coordinates of $\vartheta$, thus pulling null parameters out of the interior of the parameter space. One must therefore consider the existence of nuisance parameters on the boundary of the parameter space, and their effect on the behaviour of the test. We introduce in Corollary~\ref{coro:correction} a simple correction in the degrees of freedom of $\Lambda^\mathfrak{d}_T$ accounting for such singularities.\\

\subsection{Relation to other works and principle of the proof} The results of Theorems~\ref{tm:LRT_for_g_constant_or_not} and~\ref{thm:TCL} belong to likelihood-based parametric statistics for point processes. The main theoretical difficulty associated with this class of procedures lies in obtaining the ergodicity of the functionals of $(\lambda^T_t)$ defined within the log-likelihood. We rely to this end on some properties of the classical imbedding representation of $(\boldsymbol{N}^T_t)$, see~\cite{bremaud1996stability}, which allows for useful coupling-like arguments. We refer to~\cite{DVJ} for a typical example of such strategy at work. The technique also appears in~\cite{ClinetYoshida}, where the desired ergodicity is proven in the case $\varphi_{kl}(t) \propto \exp( - \beta t)$, as part of a systematic framework including for the parametric estimation of point processes. The case of gamma kernels $\varphi_{kl} \colon t \propto t^k \exp(- \beta t)$ deduces from~\cite{Clinet}, and the general case is covered by~\cite[Chapter 3]{kwan2023asymptotic} (\textsc{kcd}), whom generalise the embedding-based strategy to varying baselines and arbitrary kernels. \\

 The array of techniques defined within this broad corpus may be regarded as a standard approach within which falls our proof for Theorems and~\ref{thm:TCL} and~\ref{tm:LRT_for_g_constant_or_not}.  Specifically, we find that the program of~\cite{ClinetYoshida} may be satisfied in the locally stationary setting defined by~\eqref{equ:general_intensity} and~\eqref{equ:separable_kernel} via a direct transposition of the \textsc{kcd} proof scheme. We observe that a well-known decomposition of $(\lambda^T_t)$ due to~\cite{JaissonRosenbaum} may be partially extended to the locally stationary case, and some key points of the method considerably shortened as a consequence. Apart from theses technicalities, our proofs only deviate from the standard approach in the specific development required by the semi-nonparametric context of Proposition~\ref{Proposition:robust}, which mainly relies on elementary convexity arguments instead.
\section{Setting}

\subsection{Model and assumptions}
Let $(\Omega,\mathcal{F})$ be some measurable space. Throughout the rest of the article, we rely on the following construction. 

\begin{proposition}\label{prop:embedding}Let $\Prob$ be a probability measure on $(\Omega,\mathcal{F})$. Suppose that $K$ independent Poisson measures $\pi_1 \hdots \pi_K$ on $\R^2$ with unit intensity are defined within $(\Omega,\mathcal{F},\Prob)$. Denote by $(\mathcal{F}^\pi_t)$ the natural filtration of the Poisson process $(\pi_k(0,t])$. Then, for any $(\mathcal{F}^\pi_t)$-predictable process $(\lambda_{k,t})$, the process
    \begin{equation*}
        (N_{k,t})
        =
        \big( 
        \int_0^t \int_{0}^{\infty} \mathbb{1}_{\{\lambda_{k,s} \leq x \}} \pi_k(\dif x, \dif s) \big)
    \end{equation*}
    admits $(\lambda_{k,t})$ as its intensity.
\end{proposition}

Let $(\boldsymbol{N}^T_t)=(N^T_{k,t})$ be a multivariate point process. We may suppose via Proposition~\ref{prop:embedding} and classical change of measures formulas (see~\cite{Bremaudbook}) that there exists a collection $\Prob(\eta,g)$ of probability measures on $(\Omega,\mathcal{F})$ such that $(\boldsymbol{N}^T_t)$ has predictable intensity
\begin{equation*}
    \lambda^T_{k,t}(\eta,g)
    =
    \Phi_k
    \Bigl[ 
    \mu_k \big( \frac{t}{T}, \eta \big)
    +
    \sum_{l=1}^K
    \int_0^t
    g\big( \frac{t}{T} \big) \varphi_{kl}( t-s,\eta) \dif N^T_{l,s}
    \Bigr]
\end{equation*}
under $\Prob( \eta,g)$, where $ g\in C[0,1]$ and $\eta$ lives in a compact subset $\Gamma \subset \R^p$ with non-empty interior. From~\cite{DVJ}, the log-likelihood of $(\eta,g)$ is then
\begin{equation}\label{equ:loglkl_def}
    \mathcal{L}_T(\eta,g)
    =
    \sum_{k=1}^K
    \int_0^T
    \log \lambda^T_{k,s}(\eta,g) \dif N^T_{k,s}
    -
    \int_0^T \lambda^T_{k,s}(\eta,g) \dif s.
\end{equation}
Recall that we work with a single observation $(N^T_t)$ under some $\Prob(\eta^*,g^*)$, with $t \in [0,T]$. We will retain the notation $\eta^*$ for the true parameter in the sequel. For any $q \in \mathbb{N}^\star$, we also write $\lVert x \rVert_q = (\sum \lvert x_k\rvert^q)^{\nicefrac{1}{q}}$ for the $\ell^q$-norm, $\mathcal{M}_q(\R)$ for the set of $q \times q$ matrices, $\rho \colon \mathcal{M}_q(\R) \mapsto [0,\infty)$ for the spectral radius, and denote by $\lVert \cdot \rVert_{L^q}$ the $L^q$ norm $\lVert f \rVert_{L^q}^q= \int_0^\infty \lVert f(x) \rVert^q_q \dif x$. 

\begin{assumption}\label{ass:lipschitz}
    For every $k =1 \hdots K$ and any $(x,y) \in \R^2$,  $\lvert \Phi_k(x) - \Phi_k(y) \rvert \leq \lvert x-y \rvert$.
\end{assumption}

\begin{assumption}\label{ass:stability}
\begin{equation*}
    \sup_{x \in [0,1]}g^*(x)\rho\Big( 
    \int_0^{\infty}
    \varphi_{kl}(s,\eta^*) \dif s
    \Big)<1.
\end{equation*}
\end{assumption}
\begin{remark}
     Assumptions~\ref{ass:lipschitz} and~\ref{ass:stability} are typically expressed as a simultaneous condition involving the Lipschitz constant $\alpha$ of the $\Phi_k$ (see~\cite{bremaud1996stability}). Here we have set $\alpha=1$ without loss of generality so as to elude avoidable identifiability issues. We ask for similar reasons that the amplitude of the kernel be driven by $g$ only, and not by $\varphi$. Specifically, we prohibit in Assumption~\ref{ass:identifiability} certain unidentifiable parametrizations, and require that $\varphi(\cdot,\eta_1)$ and $\varphi(\cdot,\eta)$ may not be proportional when $\eta_1 \neq \eta_2$.
\end{remark}

\begin{assumption}[Identifiability]\label{ass:identifiability}
For any $c>0$ and any $\eta_1 \neq \eta_2$, there exists an interval $I \subset \R_+$ with non-null Lebesgue measure such that, for any $s \in I$,
\begin{equation*}
    \varphi(s,\eta_1)
    \neq
    c \varphi(s,\eta_2).
\end{equation*}
\end{assumption}

\begin{assumption}\label{ass:differentiability}
For any $x \in [0,1]$ and any $t \in \R_+$, the functions $\eta \mapsto \mu(x,\eta)$, and $\eta \mapsto \varphi(t,\eta)$ are thrice continuously differentiable in $\eta$ over $\Gamma$ and admit left/right derivatives at its boundary up to order $3$.
\end{assumption}

\begin{assumption}\label{ass:continuity_in_time}
For any $\eta \in \Gamma$ and any $i=1 \hdots 3$, the function $t\mapsto \partial_{\eta}^i\varphi(\cdot,\eta)$ is piecewise-continuous over $\R_+$ with at most a finite number of discontinuities. 
\end{assumption}

\begin{assumption}\label{ass:integrability}
For any $i = 0 \hdots 3$, any $q \geq 1$,  $ \sup_{\eta \in \Gamma} \lVert \partial^i_{\eta} \varphi(t,\eta) \rVert \in  L^q[0,\infty)$.
\end{assumption}
\begin{remark}
    That Assumption~\ref{ass:integrability} spans every $L^q[0,\infty)$ is technically restrictive but largely innocuous in practice. It suffices that the kernel and its derivatives in $\eta$ be integrable and bounded. This includes: \begin{itemize}
        \item exponential/Gamma kernels $\varphi_{ij} \colon t \mapsto (t+\gamma_{ij})^{k}\exp(-\beta_{ij}t)$,
        \item power law kernels $\varphi_{ij} \colon t \mapsto (t+\gamma_{ij})^{-\beta_{ij} }$,
        \item Gaussian kernels $\varphi_{ij}\colon t \mapsto \exp( -  \beta_{ij} (t-\gamma_{ij})^2)$,
    \end{itemize}
     as notable examples.
\end{remark}

\begin{assumption}\label{ass:continuity}
For any $i =0\hdots 3$, the family of functions $\{ \partial^i_{\eta} \varphi(s,\cdot) \}$ and  $\{ \partial^i_{\eta} \mu(x,\cdot) \}$, respectively indexed on $s \in \R_+$ and $x \in [0,1]$, are uniformly equi-continuous in $\eta$. 
\end{assumption}

Additionally, we forbid the process from ever resting throughout its existence.
\begin{assumption}\label{ass:no_inhibition}
    For any $k = 1 \hdots K$, $\inf_{\eta \in \Gamma} \inf_{x \in [0,1]} \Phi_k[\mu_k(x,\eta)] >0$.
\end{assumption}

Finally, in order to state the weak convergence results of the upcoming Sections, we need to introduce the family of processes $(\lambda^{x,\infty}_{k,t}(\eta,g))$ defined by
\begin{equation}\label{equ:the_stationary_intensities}
    \lambda_{k,t}^{x,\infty}(\eta,g)
    = 
    \Phi_k \Bigl[
    \mu(x,\eta)
    +
    \sum_{l=1}^K \int_{-\infty}^t
    g(x)\varphi_{kl}(t-s,\eta) \dif N^{x,\infty}_{l,s}
    \Bigr],
    \hspace{0.1cm} t \in \R,
\end{equation}
where the $(\boldsymbol{N}^{x,\infty}_t)$ are stationary Hawkes processes embedded into the same Poisson base $(\pi_k)$ as $(\boldsymbol{N}^T_t)$, each with respective intensity $
    (\lambda^{x,\infty}_t(\eta^*,g^*))
$ under $\Prob(\eta^*,g^*)$. We refer to Theorem 7 in~\cite{bremaud1996stability} for a proof of the existence of the $(N^{x,\infty}_{k,t})$.

 \subsection{Asymptotic normality of the \textsc{mle}}\label{section:parametric} Suppose for the rest of this Subsection that the reproduction rate $g$ is well specified, meaning that one has access to a collection $G_\Xi=\{ g(\cdot,\varpi) \vert \varpi \in \Xi \}$ of $C[0,1]$ indexed on some compact $ \Xi \subset \R^{d+1}$ with non-empty interior, containing the true reproduction rate $g^*$.  That is, there is some  $\varpi^* \in \Xi$ such that
 \begin{equation}\label{equ:well_param}
     g^*=g(\cdot,\varpi^*),
 \end{equation}
  and the true intensity $\lambda^T(\eta^*,g^*)$ writes $\lambda^T_t(\eta^*,g(\cdot,\varpi^*))$. It is then convenient to rephrase the model in a purely parametric fashion, to which end we introduce the parameter space
 \begin{equation*}
      \Theta
     =
     \Gamma \times \Xi = \{ \vartheta=(\eta,\varpi) \lvert \eta \in \Gamma, \varpi \in \Xi\},
 \end{equation*}
and the parametric \textsc{mle}s
\begin{align}\label{equ:parametric_mle}
    \hat{\vartheta}_T= (\hat{\eta}_T,\hat{\varpi}_T)
    &=
    \argmax_{\vartheta =(\eta,\varpi)\in \Theta}
    \mathcal{L}_T(\eta,g(\cdot,\varpi))).
\end{align}
For any $\vartheta=(\eta,\varpi) \in \Theta$, we then commit the slight abuse of notation 
\begin{equation}\label{equ:abuse}
    (\lambda^T_{k,t}(\vartheta))=(\lambda^T_{k,t}(\eta,g(\cdot,\varpi))),
\end{equation} and, likewise, write $\lambda^{x,\infty}_{k,t}(\vartheta)=\lambda_{k,t}^{x,\infty}(\eta,g(\cdot,\varpi))$ for the stationary intensities~\eqref{equ:the_stationary_intensities}. Where there is no ambiguity, we also use the notation $\Prob(\vartheta^*)=\Prob(\eta^*,g(\cdot,\varpi^*))$. The parametrized reproduction rates $g(\cdot,\varpi)$ should naturally satisfy the same regularity conditions as $\mu$ and $\varphi$ as summarized in the following Assumption.

\begin{assumption}
    The functions $g \colon x,\varpi \mapsto g(x,\varpi)$ are thrice differentiable in $\varpi$ over $\Xi$, and, for any $i=0 \hdots 3$, $\partial^i_{\varpi} g$ is uniformly equi-continuous in $\varpi$ over $\Xi$ and continuous in $x$ over $[0,1]$.  
\end{assumption}

Still with notation~\eqref{equ:abuse}, we introduce the asymptotic information $I(\vartheta)$ as the matrix
\begin{equation}\label{equ:asymptotic_information}
    I(\vartheta) = \sum_{k=1}^K\int_0^1 
     I^{(k)}(\vartheta,x)
        \dif x,
\end{equation}
where 
\begin{equation*}
   I^{(k)}_{ij}(x,\vartheta)= \mathbb{E}\left[ \frac{
        \partial_{\vartheta_i}
        \lambda^{x,\infty}_{k,0}
        (\vartheta)
        \partial_{\vartheta_j
        }\lambda^{x,\infty}_{k,0}
        (\vartheta)
        }{\lambda^{x,\infty}_{k,0}(\vartheta)}\right],
\end{equation*}
upon which is imposed the following non-degeneracy condition. 
\begin{assumption}\label{ass:fisher_is_non_degenerate}
There are some $x \in [0,1]$ and $k \in \{1 \hdots K \}$ such that $I^{(k)}(x,\vartheta)$ is definite positive.
\end{assumption}
\begin{remark} Assumption~\ref{ass:fisher_is_non_degenerate} is standard in the literature, in the sense that it simplifies into~\cite[Assumption \textsc{b}6]{OgataMLE} when $\mu\equiv 1,g \equiv 1$. While minimal general conditions on $(\mu,\varphi)$ guaranteeing the Assumption are not known, some sufficient identifiability criteria can be enunciated for specific choices of $\varphi$. See for instance~\cite{Martinez} in the exponential case.
\end{remark}
We are now ready to state our \textsc{clt}.
 \begin{theorem}[A central limit Theorem]\label{thm:TCL}
     Work under Assumptions~\ref{ass:lipschitz} to~\ref{ass:fisher_is_non_degenerate}. Grant also~\eqref{equ:well_param} and write $\vartheta^*=(\eta^*,\varpi^*)$. Then, the rescaled \textsc{mle} $\sqrt{T}(\hat{\vartheta}_T-\vartheta^*)$ converges in law under $\Prob(\eta^*,g(\cdot,\varpi^*))$ as $T \to \infty$ towards the distribution of
     \begin{equation*}
         \argmin_{h \in H} \lVert I(\vartheta^*)^{\frac{1}{2}} X -  I(\vartheta^*)^{\frac{1}{2}} h \rVert_2
     \end{equation*}
     with $X \sim \mathcal{N}(0,I(\vartheta^*)^{-1})$, $I(\vartheta^*)$ defined in~\eqref{equ:asymptotic_information}, and $H = \lim_{T \to \infty} \sqrt{T}(\Theta - \vartheta^* )$\footnote{$H$ may be appropriately defined in terms of Painlevé-Kuratowski convergence, see~\cite{Geyer}.}.
 \end{theorem}
 \begin{remark}\label{remark:actual_central_limit_th}
    The \textsc{clt}~\ref{thm:TCL} describes the distribution of the \textsc{mle} as the one of a Gaussian variable projected on $H$ along the geodesics of the Fisher-Rao metric. We have kept the statement of Theorem~\ref{thm:TCL} in such form so as to include the singularities discussed in our Introduction. When $\eta\in \mathring{ \Gamma}$, the $\argmin$ is attained at $X$, and one recovers a more conventional \textsc{clt} as detailed in Corollary~\ref{coro:true_clt}. 
 \end{remark}
 \begin{corollary}\label{coro:true_clt}
     Work under the setting of Theorem~\ref{thm:TCL}, and suppose that $\eta \in \mathring{\Gamma}$. Then,
     \begin{equation*}
         \sqrt{T}(\hat{\vartheta}_T - \vartheta^*)
         \to \mathcal{N}(0,I(\vartheta^*)^{-1})
     \end{equation*}     
     in law as $T \to \infty$ under $\Prob(\vartheta^*)$.
 \end{corollary}
 \begin{remark}
     Since we have made in Assumption~\ref{ass:continuity} the supposition that $g$ is uniformly equi-continuous in $\varpi$, the random function $g(\cdot,\hat{\varpi})$ is a consistent estimator of $g^* =g(\cdot,\varpi^*)$  under the setting of the present Section, in the sense of the uniform norm.
 \end{remark}
\subsection{Testing for time-invariance of the kernel}\label{section:nonparametric} We intend to test
\begin{equation}\label{equ:nullhyp}
    \mathcal{H}_0: \{ \textup{The reproduction rate }g \textup{  is constant} \},
\end{equation}
against the general alternative \begin{equation*}
    \mathcal{H}_1:\{ \textup{It is not}\}.
\end{equation*}
In order to stress our test's robustness to misspecification of $x \mapsto g(x,\cdot)$, we hand over the formalism of the preceding Subsection and place the next results in the semi-nonparametric setting. The process $(\boldsymbol{N}^T_t)$ still has baseline $\mu(\cdot,\eta^*)$ and kernel $\varphi(\cdot,\eta^*)$ with $\eta^* \in \Gamma \subset \R^p$, while the reproduction rate $g^*$ is left free to roam in the infinite-dimensional space $C[0,1]$. We rely on a polynomial $g^{\mathfrak{d}}(\cdot,\varpi)$ defined over the Bernstein basis
$
    B_{k,\mathfrak{d}} (t)=  {\mathfrak{d}\choose k} t^k (1-t)^{\mathfrak{d}-k}
$ of $\R_{\mathfrak{d}}[X]$. For any $x \in [0,1]$,
\begin{align}\label{equ:bernstein_basis}
g^{\mathfrak{d}}(x,\varpi)=\sum_{i=0}^{\mathfrak{d}} \varpi_i B_{i,\mathfrak{d}}(x),
\end{align}
where the weights $\varpi$ live in the non-empty interior of some compact subspace $\Xi_\mathfrak{d}$ of $\R^{\mathfrak{d}+1}_+$, consistently with the notation of Subsection~\ref{section:parametric}. For any $\mathfrak{d} \in \mathbb{N}^\star$, the basis $(B_{i,\mathfrak{d}})$ enjoys the partition property $\sum_{i=0}^{\mathfrak{d}} B_{i,\mathfrak{d}}(x)=1$, whence hypothesis~\eqref{equ:nullhyp} rephrases as
\begin{equation}\label{equ:it_is_parametric}
    \mathcal{H}_0 =\{(\eta,\varpi) \vert \varpi \in \Xi_{\mathfrak{d}}^0 \} :=\{ (\eta,\varpi) \lvert \varpi_{0} = \cdots = \varpi_{\mathfrak{d}} \}.
\end{equation}
The null hypothesis is therefore a finite-dimensional parametric sub-space of the general model, and a fitting test statistic is the likelihood ratio as defined in~\ref{def:LRS}.

\begin{definition}\label{def:LRS}
The likelihood ratio statistic (\textsc{lrs}) for testing~\eqref{equ:it_is_parametric} is defined by
    \begin{equation*}
    \Lambda^\mathfrak{d}_T
    =
    2 \{ 
    \sup_{ \eta \in \Gamma, \varpi \in \Xi_{\mathfrak{d}}^0 } \mathcal{L}_T(\eta,g(\cdot,\varpi))
    -
    \sup_{ \eta \in \Gamma,C>0  } 
    \mathcal{L}_T(\eta,C))
    \}.
\end{equation*}
\end{definition}

As mentioned in Section~\ref{section:introduction} and made clear in~\eqref{equ:it_is_parametric}, the law of $\Lambda^\mathfrak{d}_T$ under $\mathcal{H}_0$ is a parametric matter. Theorem~\ref{tm:LRT_for_g_constant_or_not} is then within the range of consequences of Theorem~\ref{thm:TCL} and its Corollary~\ref{coro:true_clt}.

\begin{theorem}\label{tm:LRT_for_g_constant_or_not}
    Under Assumptions~\ref{ass:lipschitz} to~\ref{ass:differentiability}, for any $\gamma \in \mathring{\Gamma}$ and any $C>0$,
    \begin{equation*}
        \Lambda^{\mathfrak{d}}_T \xrightarrow[T \to \infty]{\mathcal{L}(\Prob(\eta,C))}
        \chi^2(\mathfrak{d})
    \end{equation*}
\end{theorem}

In the specific case where $g^* \in \R_{\mathfrak{d}}[X]$, the distribution of $\Lambda^{\mathfrak{d}}_T$ under the alternative files again under the parametric setting. The consistency of the test then deduces from that of the \textsc{mle}. In the general case where $g^* \in C[0,1]$, one instead expects a bias to subsist in the estimation of $\hat{g}^\mathfrak{d}_T$, as we have not allowed $\mathfrak{d}$ to increase with $T$. However, the consistency of the \textsc{mle} is not necessary for the consistency of the test. The \textsc{mle} over $\Gamma \times \Xi^{\mathfrak{d}}$ needs only be marginally better than the \textsc{mle}  over  $\Gamma \times \Xi^{\mathfrak{d}}_0 $ -- sufficiently so that a gap appears within $\frac{1}{T} \Lambda_T^{\mathfrak{d}}$, resulting in a diverging \textsc{lrs}. We find in Proposition~\ref{Proposition:robust} that, provided $\mathfrak{d}$ is above a certain threshold depending only on $g^*$ but not on $T$, the likelihood ratio test remains consistent for any non-constant $g \in C[0,1]$. We need only slightly strengthen Assumption~\ref{ass:no_inhibition}.

\begin{assumption}[Convexity]\label{ass:convexity}
    There is some $\epsilon>0$ such that, for any $k = 1 \hdots p$ and $x \geq 0$, $\Phi_k'(x)>\epsilon$.
\end{assumption}

\begin{proposition}\label{Proposition:robust}
    Work under Assumptions~\ref{ass:lipschitz} to~\ref{ass:identifiability}. Let $g^*$ be a non-constant function from $[0,1]$ to $\R_+$. Then, there exists $d^*>0$ depending only on $\eta^*$ and $g^*$ such that, for any $\mathfrak{d}>d^*$, $\Lambda_T^{\mathfrak{d}} \to \infty$ in probability as $T \to \infty$ under $\Prob(\eta^*,g^*)$.
\end{proposition}

\begin{remark}\label{remark:explicit_d_in_special_case}
    In all generality, $d^*$ is intractably tied to $(\eta^*,g^*)$. In the simplified experiment where $\eta$ is known and $g^*$ is $\mathcal{K}$-Lipschitz-continuous for some $\mathcal{K}>0$, it is however possible to show the existence of a constant $C_{\mu,\Phi}>0$ depending only on $\Phi$ and $\mu$ such that $\Lambda^\mathfrak{d}_T \to \infty$ as soon as
    \begin{equation*}
        \mathfrak{d} \geq C_{\mu,\Phi}\mathcal{K}^{-1}
        \inf_{\{C>0\}} \lVert g^*- C \rVert^4_{L^{2}[0,1]}.
    \end{equation*}
    This translates the intuition that the power of the test should be an increasing function of the distance from $g$ to the constant functions over $[0,1]$, as we verify empirically in Section~\ref{section:num}.
 \end{remark}
 An important implicit Assumption of Theorem~\ref{tm:LRT_for_g_constant_or_not} is that $\eta$ lies in the interior $\mathring{\Gamma}$ of its parameter space. The condition is acceptable in lower dimensional settings (the case $K=1$ covering a large array of applications), or in dense contexts where one expects the interaction matrix $\int_0^{\infty} \varphi_{kl}(s) \dif s$ not to contain any null value. In general however, a certain level of sparsity is expected, in the sense that $\int_0^{\infty} \varphi_{kl}(s) \dif s=0$ for some index coupes $(k,l)$, a consequence thereof is the existence of nuisance parameters on the boundary of $\Theta$. Singularities then arise in the distribution of the \textsc{mle} and therefore in that of the \textsc{lrs} $\Lambda^{\mathfrak{d}}_T$. They must be accounted for via a correction in the degrees of freedom in Theorem~\ref{tm:LRT_for_g_constant_or_not}. The correction depends on the number of zeros in the \textsc{mle} and we introduce the random counter
 \begin{align*}
     \hat{k}_T=\sum_{i}^d \mathbb{1}_{\{\hat{\vartheta}_{i,T}>0 \} \cap \{\hat{\vartheta}^0_{i,T}=0 \} },
 \end{align*}
 which is simply the number of estimates that are null in the constant model, and positive in the polynomial one. The following correction readily deduces from~\cite[Annex]{autocite}.
 \begin{corollary}\label{coro:correction}
     Grant the Assumptions of Theorem~\ref{tm:LRT_for_g_constant_or_not}. Then, for any $\eta \in \Gamma$, under $\Prob(\eta^*,g^*)$,  any $k \in \mathbb{N}$, for any sufficiently large $n$,
     \begin{equation*}
         \Prob\big[ \Lambda_T >x\lvert \hat{k}_T=k \big] \leq S_{\mathfrak{d}+k}(x)
     \end{equation*}
     where $S_i$ is the survival function of a $\chi^2(i)$ distribution.
 \end{corollary}

 \begin{remark}
     Our choice of a likelihood ratio test is motivated in part by that it does not require estimating the Fisher Information of the model, which size grows as $K^4$ where $K$ is the dimension of the process. This sidestep is naturally not entirely free of cost as Corollary~\ref{coro:correction} indicates a moderate loss of power is incurred for each additional null nuisance parameter. In this respect, our test is a conservative one. In the univariate case, where other choices may find one's preference, we refer to~\cite[Proposition 6.1]{Clinet} for confidence interval construction methods, which provide the sufficient basis for conducting a Wald or score test instead.
 \end{remark}

Corollary~\ref{coro:correction} ends our partial excursion out of the parametric domain. Before we proceed to the proofs, we offer some numerical results both simulated and applied.

\section{Numerical results}\label{section:num}
\subsection{Convergence of the kernel time-dependence test} We provide in figure~\ref{fig:test_verification} an empirical illustration of Theorem~\ref{tm:LRT_for_g_constant_or_not} and Proposition~\ref{prop:consistent}.  The behaviour of $\Lambda^{\mathfrak{d}}_T$ under the null is fist considered, to which end we simulate $N=40000$ trajectories of a Hawkes process with intensity
\begin{equation}\label{equ:null_sim}
    \lambda_t = \mu + \int_0^t e^{-2(t-s)} \dif N_s, \hspace{0.1cm} t \in [0,T],
\end{equation}
where $\mu=1$. We then turn our attention towards the alternative, and estimate the power function of the test for several different values of $\mathfrak{d}$ upon $4000$ simulations of a locally stationary Hawkes process, again with constant baseline $\mu=1$, and intensity
\begin{equation}\label{equ:alternative_sim}
    \lambda^T_t = \mu + \int_0^t (\gamma+ \alpha_0 \sin(\alpha_1 \frac{t}{T})) e^{-2(t-s)} \dif N^T_s, \hspace{0.1cm} t \in [0,T].
\end{equation}

The parameter $\alpha_0$, which is proportional to the uniform distance to $t \mapsto1$ of the sinusoidal reproduction rate in~\eqref{equ:alternative_sim}, is used as a proxy for the distance to the null hypothesis. The power of the test is thus understood as as a function of $\alpha_0$.
\begin{figure}[H]
    \centering
    \includegraphics[width=0.25\linewidth]{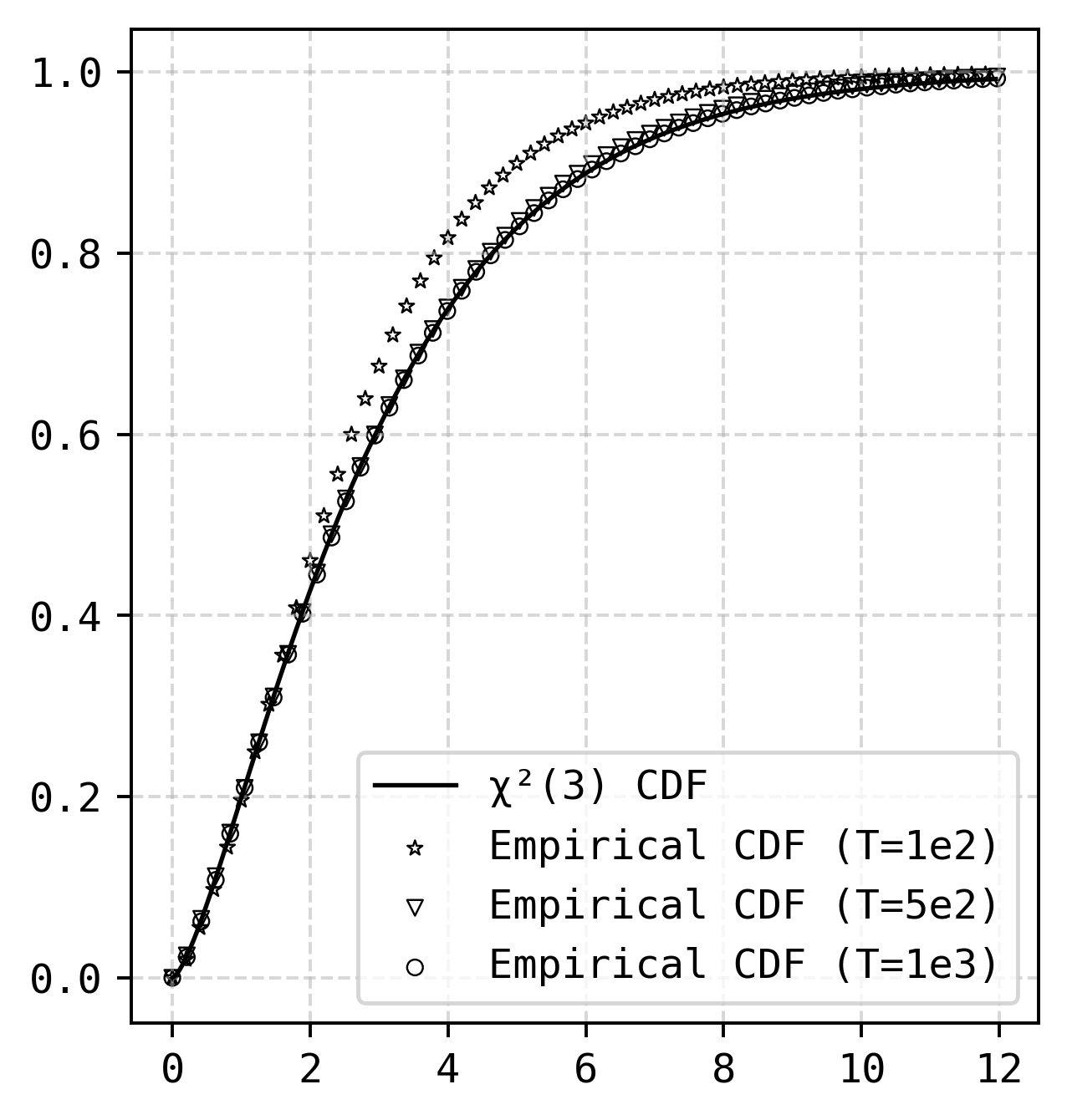}
    \includegraphics[width=0.25\linewidth]{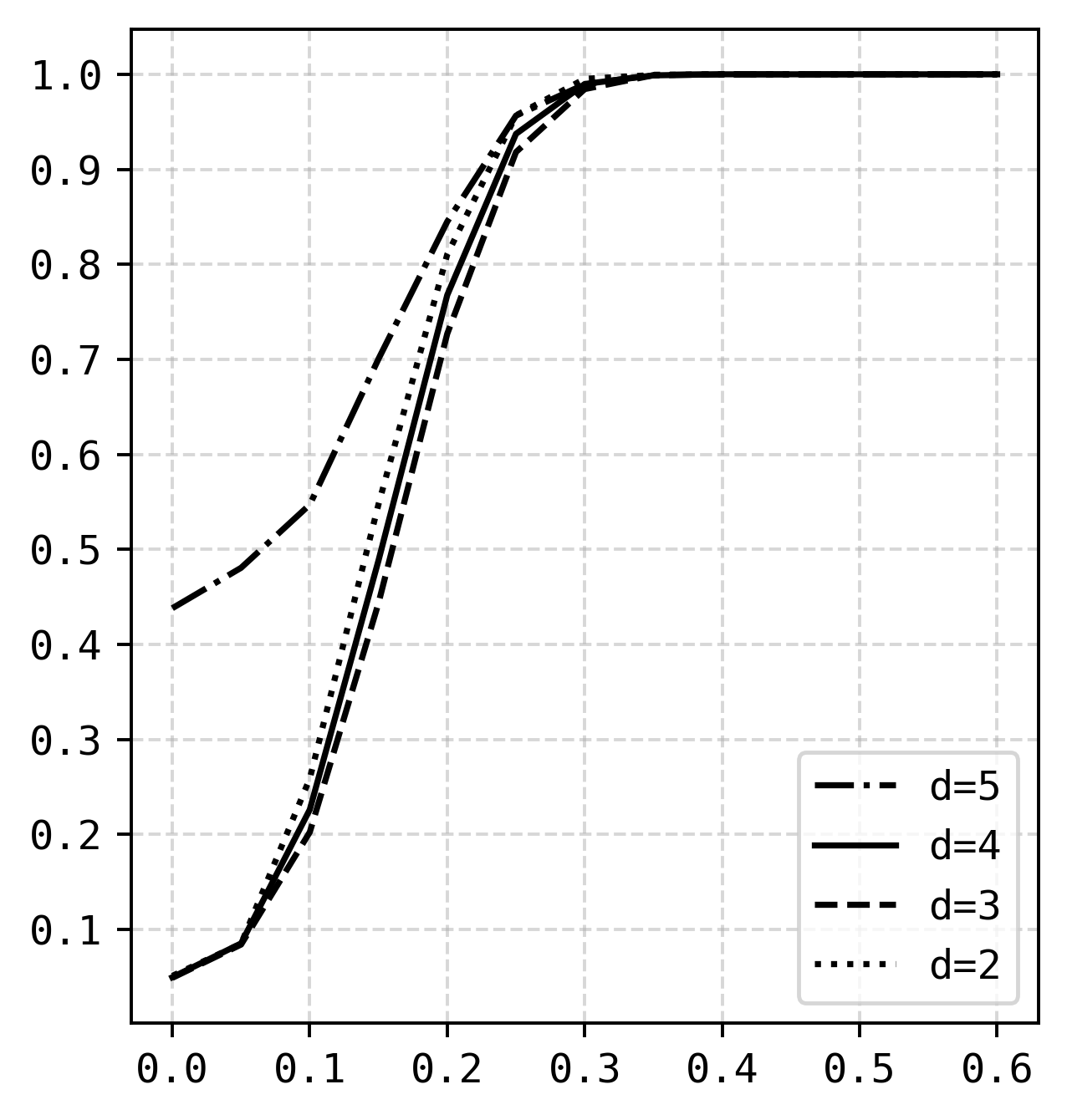}
    \caption{Distribution under the null and the alternative of the test. \textbf{Left}: empirical cumulative distribution function under the null with $\mathfrak{d}=3$ for three different values of $T$ ($N=40000$ simulations of the model~\eqref{equ:null_sim}). \textbf{Right}: empirical power function as a function of $\alpha_0$ and $4$ different values of $\mathfrak{d}$, $\gamma=\alpha_1=1$, and $T=10000$, over successive sets of $N=4000$ simulations of the model~\eqref{equ:alternative_sim}.}
    \label{fig:test_verification}
\end{figure}

Some practical remarks are in order. We observe that, at least for~\eqref{equ:alternative_sim}, the degree $\mathfrak{d}$ has little effect on the consistency of the test. This indicates that the smallest admissible value for $\mathfrak{d}$ as described in Proposition~\ref{prop:consistent} is rather low in practice. In contrast, the distribution of $\Lambda^{\mathfrak{d}}_T$ under the null depends non-linearly on $\mathfrak{d}$. As long as the standard deviation of the parameters' estimate remains below the size of $\Theta$, the test reaches its expected asymptotic distribution, as seen in the case $\mathfrak{d}=3$ on the left-hand side of Figure~\ref{fig:test_verification}. When $\mathfrak{d}$ is too high relatively to $T$ however, the estimates for $\varpi$ spill out of the boundaries of $\Xi$. A Gaussian limit fails to arise in the \textsc{mle}, and a $\chi^2(d)$ limit may not arise in the \textsc{lrs}.  This results in the singularities observed at $\mathfrak{d}=5$ on the right-hand side of Figure~\ref{fig:test_verification}. We have stressed in~\ref{section:nonparametric} that the convergence of $\Lambda^\mathfrak{d}_T$ only loosely depends on that of the \textsc{mle}. In regard of Figure~\ref{fig:test_verification}, this does not absolve one from exercising elementary caution, by which we mean verifying that the \textsc{mle} takes its values within admissible parameters, and otherwise lowering $\mathfrak{d}$.  In view of the right-hand side of Figure~\ref{fig:test_verification} this is unrestrictive in practice. We also reproduce in Figure~\ref{fig:g_estimation} the resulting estimation for $g$ so as to highlight the performance of the polynomial estimator $g^\mathfrak{d}$.

\begin{figure}[H]
    \centering
    \includegraphics[width=0.26\linewidth]{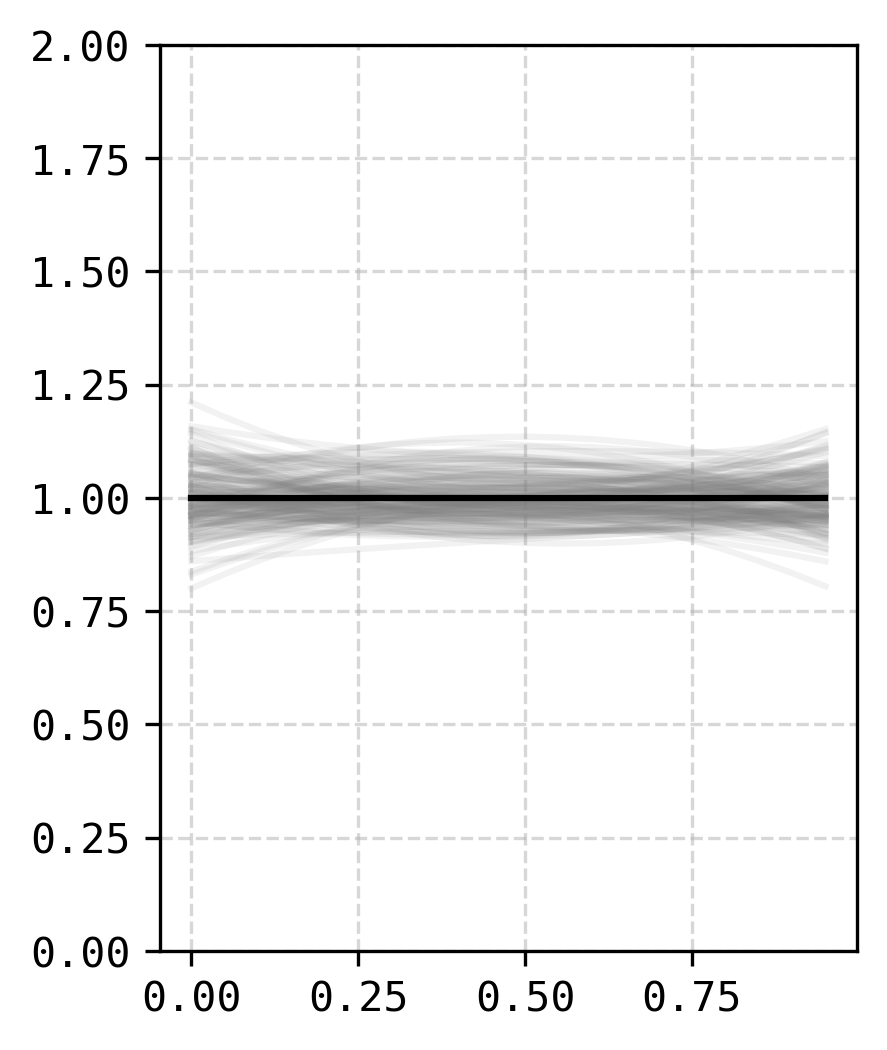}
    \includegraphics[width=0.26\linewidth]{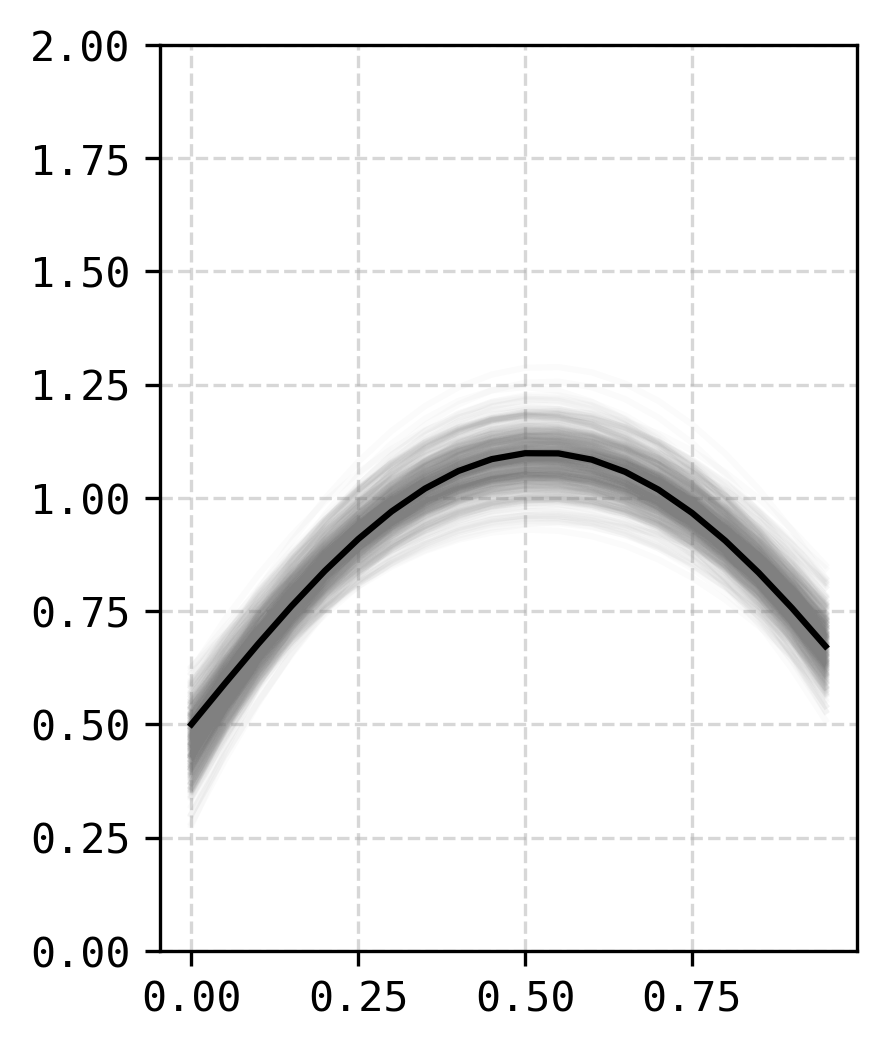}
    \caption{ \textsc{mle} $\hat{g}^\mathfrak{d}_T$ for $g=1$ as in~\eqref{equ:null_sim} (left) and a sinusoidal $g$ as in~\eqref{equ:alternative_sim} with $alpha_0=0.6,\alpha_1=5.0$ (right), over simulated data. Simulations performed via classical thinning algorithms, with $T=5000$, $\beta=2$, and $\mu=1$ in all cases.  $N=500$ estimates in light grey, true reproduction rate in black.}
    \label{fig:g_estimation}
\end{figure}

The results obtained in Figure~\ref{fig:g_estimation} are consistent with the claims of the preceding Sections. With a constant $g^*=1$, The \textsc{clt}~\ref{thm:TCL} describes the variations of the estimator around the true value. For smooth enough fluctuations of $g^*$, the estimator provides a convincing approximation of the true reproduction rate, although its convergence rate is not described by Theorem~\ref{thm:TCL} anymore. A fully nonparametric estimator with time-dependent degree is beyond the aims of the present paper, which requires a fixed $\mathfrak{d}$ for testing purposes. In view of Figure~\ref{fig:g_estimation}, we are still compelled to gauge, at the exploratory level, the performance of the polynomial \textsc{mle} $\hat{g}^{\mathfrak{d}}_T$ as $\mathfrak{d}$ and $T$ both increase. We find in Figure~\ref{fig:enterfull_nonparam} that the naive polynomial \textsc{mle} indeed provides a good candidate for a consistent estimator.
 
\begin{figure}[H]
    \centering
    \includegraphics[width=0.25\linewidth]{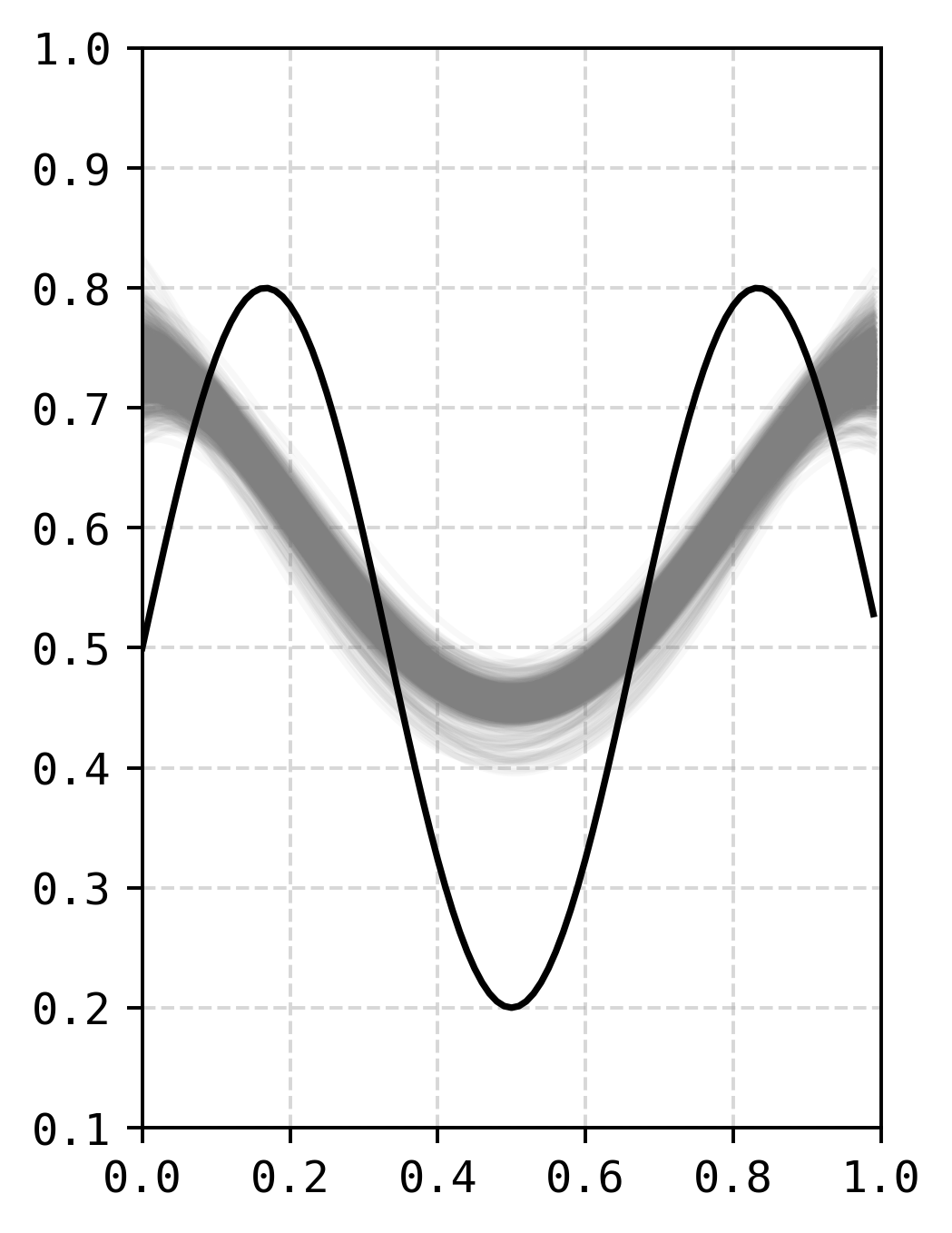}
    \includegraphics[width=0.25\linewidth]{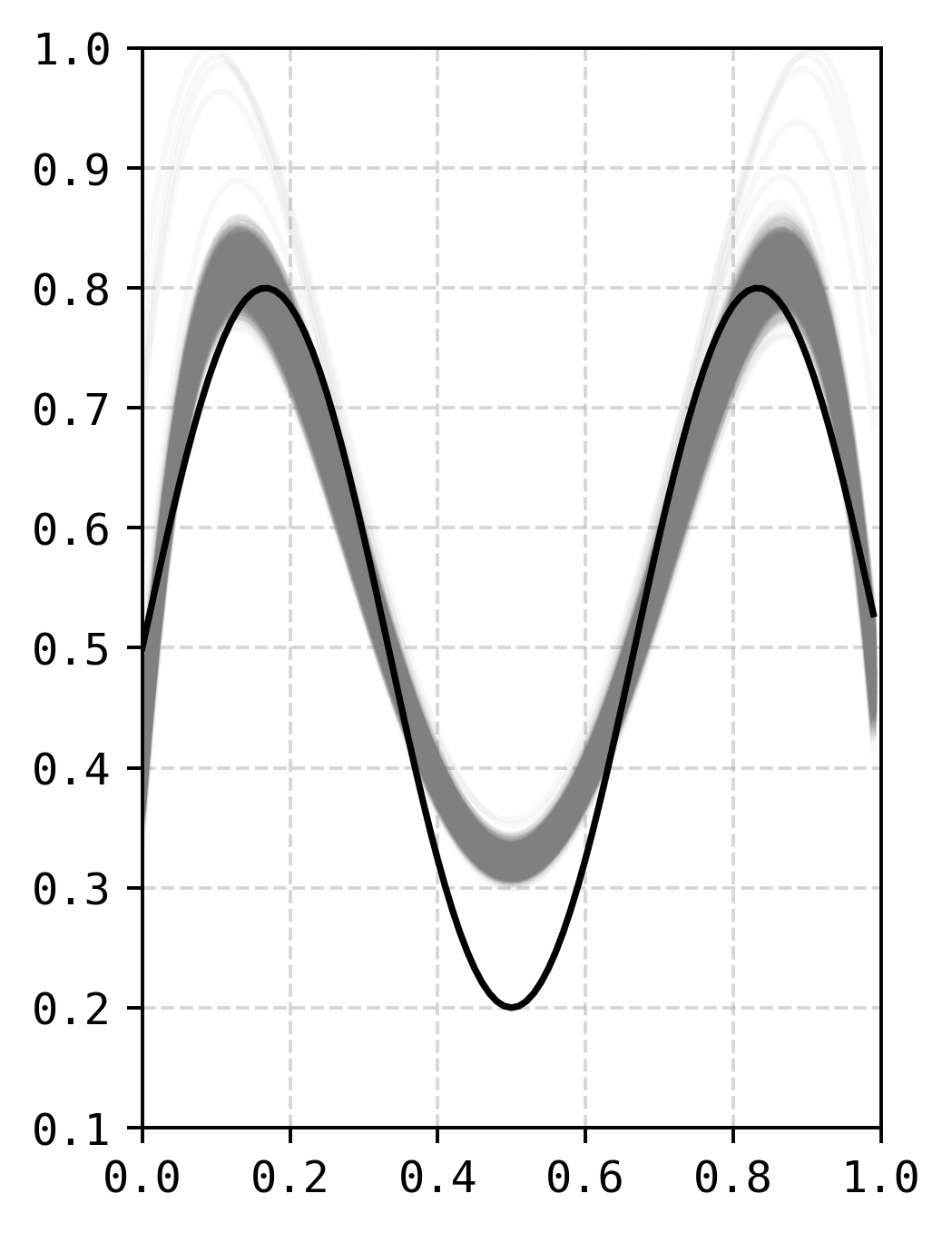}
    \includegraphics[width=0.25\linewidth]{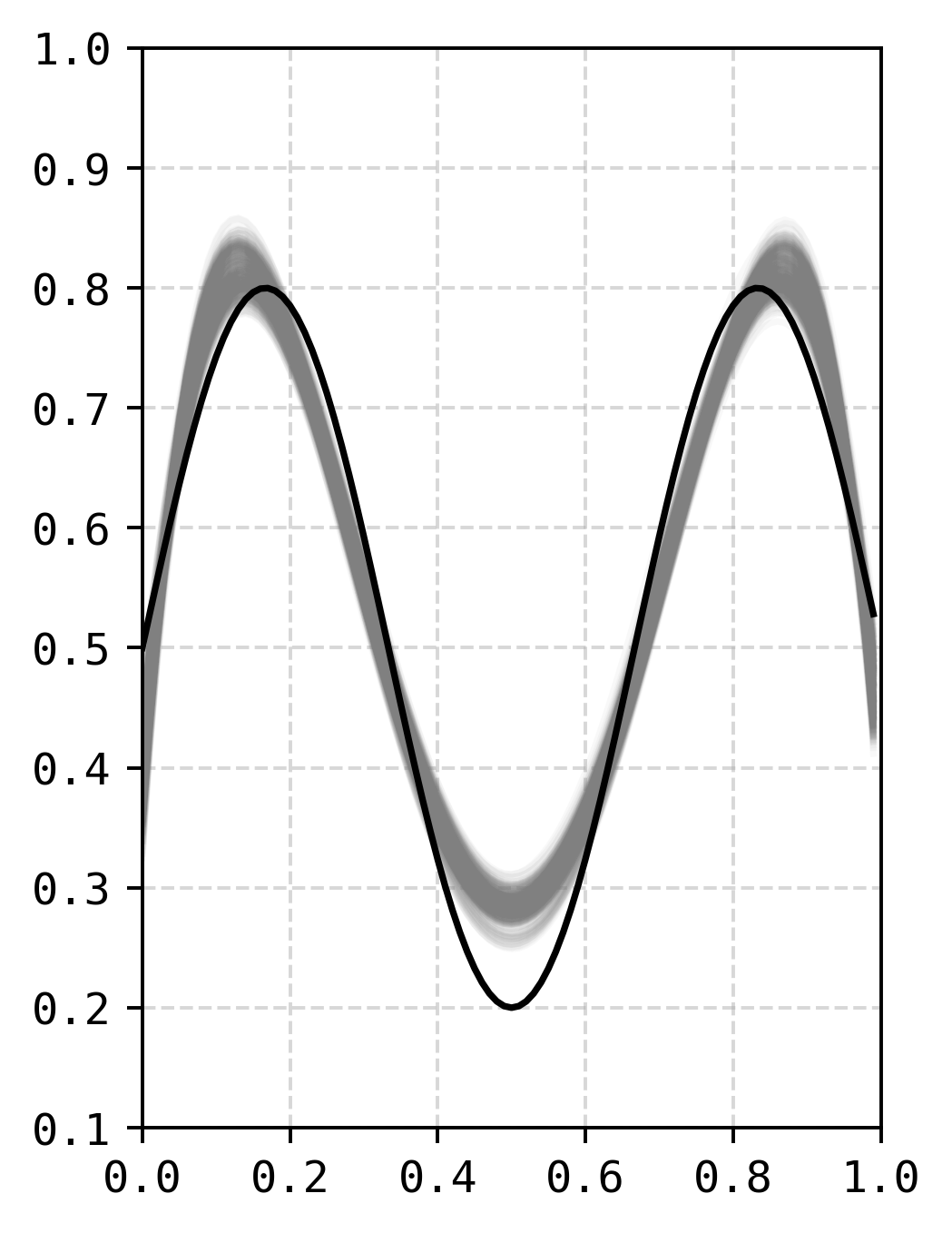}
    \caption{In grey: \textsc{mle} for $g$. In black: true reproduction rate, with intensity as in~\eqref{equ:alternative_sim}, $\mu=1,\beta=2$, $\gamma=0.5,\alpha_0=0.3,\alpha_1=3\pi$. $N=1000$ simulations via thinning with the pair $(\mathfrak{d},T)$ successively  set to $(4,10000)$ (\textbf{left}), $(6,20000)$ (\textbf{middle}), $(8,40000)$ (\textbf{right}).}
    \label{fig:enterfull_nonparam}
\end{figure}
Finally, we remark in Figures~\ref{fig:g_estimation} and~\ref{fig:enterfull_nonparam} that the \textsc{mle} correctly identifies global extrema even as large biases may subsist in uniform distance. This opens some interesting perspective regarding change-point problems, as the test described in~\ref{section:nonparametric} is compatible with certain detection procedures, notably the \textsc{cusum} algorithm as applied to Hawkes processes in~\cite{chevalier2023uncoveringmarketdisorderliquidity}. 
\section{Application to intraday power markets order flow}
\subsection{Context.} Limit order books (\textsc{lob}) are dynamic registry and matching systems used to centralise orders in electronic markets.  Liquidity is proposed in continuous time $t \in [0,T]$ either at the ask$^a$ or bid$^b$ side of the \textsc{lob} through limit orders ($\textsc{l}^a_t$,$\textsc{l}_t^b$), then to be matched by market orders ($\textsc{m}^a_t$,$\textsc{m}_t^b$), or simply cancelled ($\textsc{c}^a_t$,$\textsc{c}_t^b$). Point process models can naturally render the event-time structure of an \textsc{lob} and have been extensively used in that capacity. A detailed account of such models is found in~\cite[section 4]{ClinetYoshida}. In particular,~\cite{Toke2011} proposes a Hawkes model of the process $(\textsc{l}_t,\textsc{m}_t)=(\textsc{l}_t^a+\textsc{l}_t^b,\textsc{m}_t^a+\textsc{m}_t^b) $, finding an absence of $\textsc{l}\to \textsc{m}$ interactions on French equity markets. Hawkes models are also popular tools for the statistical study of energy markets, specifically power markets -- see~\cite{QEF_hawkes} for a review.~\cite{KRAMER2021105186} for instance estimated a stochastic baseline Hawkes process on $2015$ German intraday power market \textsc{lob}  data to isolate exogenous contributions 
 to the dynamic of $(\textsc{b}_t,\textsc{s}_t)=(\textsc{l}_t^b+\textsc{m}_t^b,\textsc{l}_t^a+\textsc{m}_t^b)$.  Here we consider \textsc{lob} data for $2023$ on the same market as~\cite{KRAMER2021105186} and implement a complete order flow model for the  $4$-dimensional process $(\textsc{F}_t)=(\textsc{l}^a_t,\textsc{l}^b_t,\textsc{m}^a_t,\textsc{m}^a_t)$.
\subsection{Dataset \& Model} The intraday power market serves the re-hedging and re-balancing needs induced by short-term\footnote{Occurring after the day-ahead auction has settled, that is.} variations in power demand and production. Electricity is traded via quarter-hourly to hourly future contracts, each originated at $15:00$ the day before delivery and expiring minutes before its delivery starts at time $T$. At $T^-=T-1$h, the delivery zone shrinks, thereby changing the nature of the product. We refer to the \textit{de facto} expiry $T^-$ as a \textit{virtual close}. Liquidity varies considerably as the session unfolds, with the bulk of trading activity concentrated in the hours preceding the virtual close.  This non-stationarity is typically accounted for via a (possibly stochastic) time-dependent baseline as in~\cite{KRAMER2021105186} or~\cite{DeschatreGruet}. Recalling the cluster representation of~\cite{Oakes}, such model choices will attribute order flow variations to exogenous factors alone. This is to say informed order flow emanating from fundamentals-motivated agents drives the entire fluctuation in liquidity, while traders remain impassive to the expiry looming closer. We apply the framework developed in Sections~\ref{section:parametric} and~\ref{section:nonparametric}  to assess whether one should also consider changes in participants' engagement with the intraday power market as a driving factor in liquidity.\\

The data consists of $21$ trading sessions of the German intraday power markets for the hourly future with delivery from $18:00$ to $19:00$ (\textsc{gmt}$+1$). The retained orders are registered on the day of delivery, as we leave the last hour of trading out of the sample to avoid modelling re-definitions of the delivery terms. Each trajectory of $(\textsc{f}_t)=(\textsc{L}^a_t,\textsc{L}^b_t,\textsc{M}^a_t,\textsc{M}^b_t)$ thus consists in $17$ hours of \textsc{lob} records from midnight to the virtual close. The point process $(\textsc{f}_t^T)$, is endowed with the intensity
\begin{equation*} 
    \lambda^{T^*}_{i,t}(\vartheta)
    =
    \mu\big(\frac{t}{T^-},\omega^{\mu}\big)
    +
    g\big(\frac{t}{T^-},\omega^g\big)
    \sum_{j \in I}
    \int_0^t
    \alpha_{ij}
    \beta e^{ - \beta (t-s)}\dif N^T_{j,s},
    \hspace{0.2cm}
    t \in [0,T^{-}],
\end{equation*}
where $\Prob(\vartheta)$ $\mu(\cdot,\omega^\mu)$ and $g(\cdot,\omega^g)$ are polynomials with coefficients $\omega^{\mu}=(\omega^{\mu}_i)$ and $\omega^{g}=(\omega^{g}_i)$ in the Bernstein basis of $\R_4[X]$ and we impose the interaction structure 
\begin{equation*}
(\alpha_{ij})=
\bordermatrix{%
    & \textsc{L}^a      & \textsc{L}^a      & \textsc{M}^a     & \textsc{M}^b \cr   
\textsc{L}^a   &  a & b & c & 0\cr 
\textsc{L}^b   & b & a & 0 & c\cr
\textsc{M}^a   & d & 0 & e & 0\cr
\textsc{M}^b   & 0 & d & 0 & e\cr
}
\end{equation*}
upon $(\alpha_{ij})$. It may be noted that the model is here unidentifiable without further constraints, hence we impose that $\varpi_0=1$. This does not modify our theoretical setting as the resulting model spans the same class of intensities. Our question regarding the agents behaviour can then be recast as the fluctuation test of Section~\ref{section:nonparametric}, which hypothesis simplifies into $\varpi_1= \cdots = \varpi_{\mathfrak{d}}=1$ under the preceding constraint, and no further modifications need be considered.

\subsection{Liquidity profile of the intraday power market}

In figure~\ref{fig:MLE_g_LOB_data} is reproduced the estimated reproduction rate $\hat{\rho} \colon x \in [0,1] \mapsto g( \hat{\eta}_T, x) \rho(\hat{\alpha}_T)$ over a selection\footnote{Data for the $21^{\textup{st}}$ of March $2023$ is skipped due to a half hour of missing data.} of $2023$ trading sessions

\begin{figure}[H]
    \centering
    \includegraphics[width=0.42\linewidth]{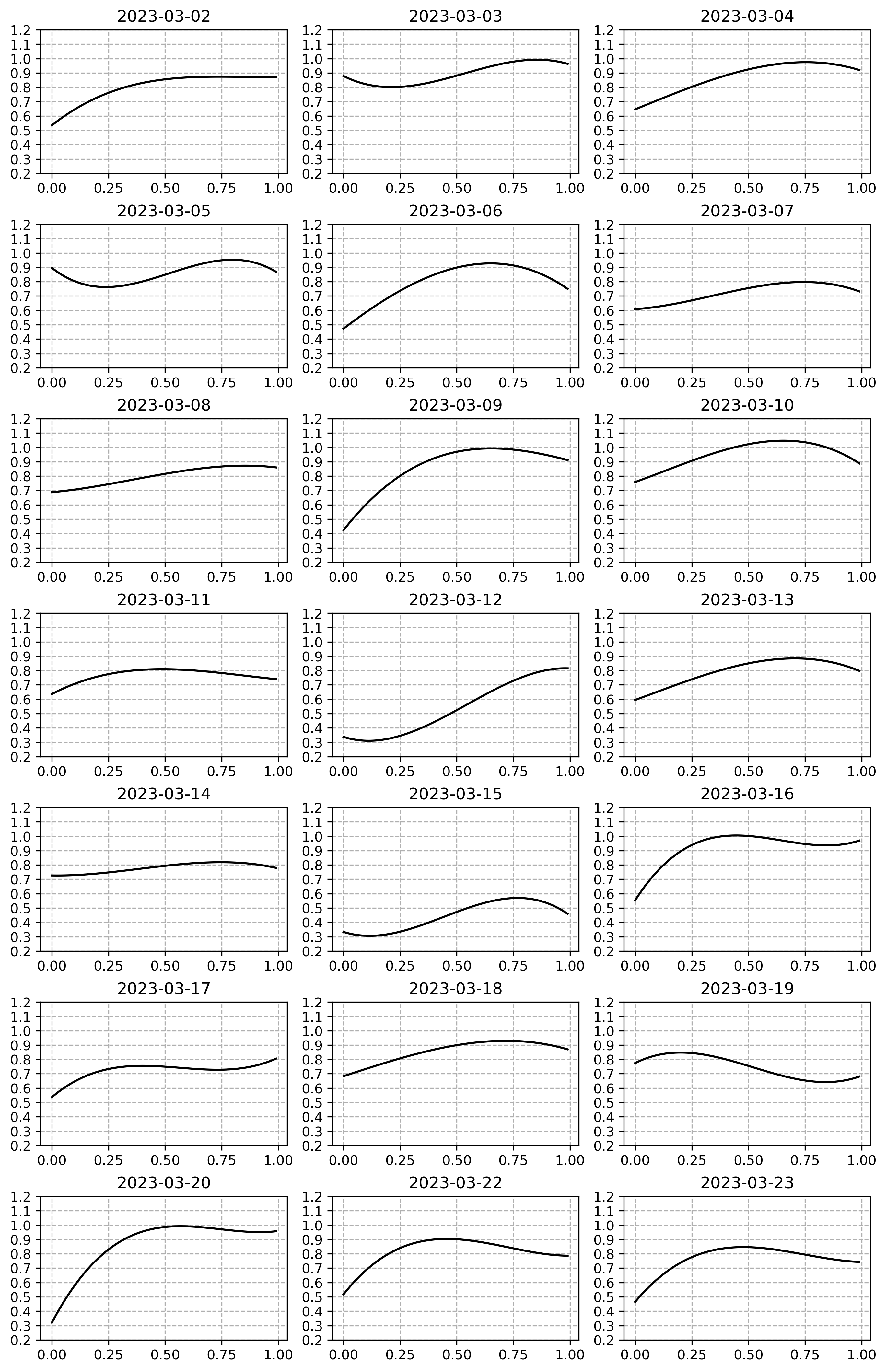}
    \caption{Estimated reproduction rate $\hat{\rho}$ on $\textsc{l} \nicefrac{}{} \textsc{m}$ orders arrivals from midnight to $T-1$h before delivery on $21$ trading sessions of March $2023$ for the intraday hourly power future with delivery starting at $T=18$h.}
    \label{fig:MLE_g_LOB_data}
\end{figure}

The null hypothesis that $g$ is constant is rejected at the $95\%$ confidence level for all but one of the trajectories, with the exception arising at the $03/14$ trading session. Though the precise pattern for $g$ appears idiosyncratic to each session a recurring motif can observed. The endogeneity rate tends to increase, reach an apex a few hours away from delivery and decrease afterwards. This would be consistent with traders re-hedging their positions as information on physical conditions at the close flows in, and these among them less prone to house physical risk progressively retreating from the market as delivery nears by.

\section{Conclusion}

\subsection{Findings and perspectives} We have shown in Section~\ref{section:parametric} that the conventional maximum likelihood estimator for point processes may be indifferently applied to a class of non-stationary non-linear multivariate Hawkes processes with time dependent reproduction. Our first result in Theorem~\ref{thm:TCL} takes the form of a central limit Theorem, in which we recover the weak convergence of the \textsc{mle} towards a Gaussian distribution. Our second result proves the convergence of the associated likelihood ratio test and extends it consistency over the entire space of $C[0,1]$ functions.\\

The asymptotic variance of the \textsc{mle} defined in Theorem~\ref{thm:TCL} is described by a uniformly weighted average of the stationary asymptotic Fisher Information matrices described in~\cite{OgataMLE}. A natural research direction suggested by the numerical results of Section~\ref{section:num} is in extending the convergence of $\hat{g}^\mathfrak{d}_T$, in the spirit of~\cite[Theorem 5.3]{Clinet}. It is perhaps relevant to note that replacing the Bernstein basis in favour of the discontinuous Haar-like basis formed of scaled indicator functions yields a theoretical setting compatible with the framework of~\cite{Clinet}. Should Theorem~\ref{thm:TCL} extend in such direction indeed, one may then ask how the global naïve \textsc{mle} compares to its local counterpart when applied to the estimation of the integrated parameter $\int_0^1 g(x) \dif x$. Again, we recall that a fixed degree $\mathfrak{d}$ is needed to obtain the results of Section~\ref{section:nonparametric}, hence we leave such problems to future research.\\

Additionally, it may also be noted that Theorem~\ref{coro:true_clt} extends the literature in the non-linear direction, in the sense that we have allowed for a large array of admissible activation function, beyond the usual case $\Phi_k(x)=x$. The main limitation of our results lie in that we have prohibited inhibitive interactions in the structure of the process. This restricts the relevance of introducing non-linear $\Phi_k$. Our examples are accordingly all set in the linear case. Some thorough numerical evidence of the convergence of the \textsc{mle} for inhibitive Hawkes processes are found in~\cite{Martinez}. Its theoretical verification constitutes another possible line of research. The main obstacle to the inhibitive extension lies in proving the existence of moments of sufficiently high order for the inverse intensity $(\lambda_t^{-1})$. This escapes the techniques in~\cite{ClinetYoshida} and the standard approach upon which we have relied, and may therefore warrant the development of a specific methodology. \\

In the context of intraday power markets, we find via our testing procedure some statistical evidence of fluctuations in the German market's endogeneity rate. This hints at a time-dependent participation and a possible segmentation in the typology of market agents, consistently with recent finding on in equity markets in~\cite{roueff2} and~\cite{chevalier2023uncoveringmarketdisorderliquidity}. A next step in this direction may for instance consider how this applies to optimal execution with Hawkes-based models as in~\cite{alfonsi}.

\section{Proofs}
\subsection{Organisation of the proofs and notation}\label{section:organisation}
As outlined in the introduction, we intend to satisfy the program of~\cite{ClinetYoshida}, whose notation we mimic. For any $k \in \mathbb{N}^\star$,  we denote by $C_{p}(\R^k)$ the continuous functions from $\R^k$ to some $\R^d$ of polynomial growth, and by $C_{\uparrow}(\R^k)$ the continuously differentiable functions of $C_p(\R^k)$ with gradient in $C_p(\R^k)$. The sufficient conditions for the convergence of the \textsc{mle} then express in terms of three sets of sufficient conditions. 
Firstly, some regularity conditions, which Assumption~\ref{ass:continuity} is tailored to fit. Secondly, the existence of uniform moment bounds for the intensity, found in Proposition~\ref{prop:moments}. Third and finally, the ergodicity of the $(\lambda^T_t(\eta,g))$ in the sense of Proposition~\ref{prop:ergodicity}. Apart from the proof of Proposition~\ref{prop:consistent}, the rest of this article is thus set in a parametric context. We re-employ the notation~\eqref{equ:abuse} from Section~\ref{section:parametric} so as to reflect this setting and lighten the statement of the following results.
\begin{Proposition}\label{prop:moments}
    With notation~\eqref{equ:abuse} and Assumptions~\ref{ass:lipschitz} to~\ref{ass:integrability}, for any  $i = 0\hdots 3$ and any $q \in \mathbb{N}$, 
    \begin{equation*}
        \sup_{T>0} \sup_{t \in [0,T]}
        \mathbb{E}\Bigl[
       \sup_{\vartheta \in \Theta}
       \lVert \partial_{\vartheta}^i \lambda_t^T
      (\vartheta) \rVert_q^q \Bigr] < \infty.
    \end{equation*}
\end{Proposition}

Proposition~\ref{prop:moments} corresponds to Assumption \textbf{[\textsc{a}2]} in~\cite{ClinetYoshida}, and the following ergodic result to Assumption \textbf{[\textsc{a}3]}. In its statement, for any smooth enough function $f \colon \Theta \mapsto \R$, we denote by $\partial_{\vartheta}^{\otimes 2} f$ the hessian matrix of $f$.

\begin{Proposition}\label{prop:ergodicity}
     Let $\psi \in C_{\uparrow}(\R^K \times \R^K \times \R^K(d+p+1) \times \R^{K (d+p+1)^2}$. For any $\vartheta \in \Theta$, let 
     \begin{align*}
         Y^T_t(\vartheta)
         &=
         \big( \lambda^T_t(\eta^*,g^*) ,\lambda^T_t(\vartheta),\partial_{\vartheta}\lambda^T_t(\vartheta),\partial_{\vartheta}^{\otimes 2} \lambda^{T}_t(\vartheta)\big)
         \\
         Y^{x,\infty}_t(\vartheta),
         &=
         \big( \lambda^{x,\infty}_t(\eta^*,g^*) ,\lambda^{x,\infty}_t(\vartheta),\partial_{\vartheta}\lambda^{x,\infty}_t(\vartheta),\partial_{\vartheta}^{\otimes 2} \lambda^{x,\infty}_t(\vartheta)\big).
     \end{align*}
     Then, under Assumptions~\ref{ass:lipschitz} to~\ref{ass:integrability}, for any $u \in [0,1]$,
    \begin{equation*}
        \sup_{\vartheta \in \Theta}
        \Big\lVert 
        \frac{1}{T}
        \int_0^{Tu}\psi \big(Y^T_t(\vartheta)\big)
        \dif t
        -
        \int_0^u 
        \mathbb{E}\bigl[\psi \big(Y^{x,\infty}_0(\vartheta)\big)\bigr] \dif x
        \Big\rVert
        \to 0
    \end{equation*}
    in $L^1(\Prob(\eta^*,g^*))$ as $T \to \infty$.
\end{Proposition}

Grant Propositions~\ref{prop:moments} and~\ref{prop:ergodicity}, which proofs are postponed to Sections~\ref{section:proof_of_moments} and~\ref{section:proof_of_ergodicty}. Work in the setting of Section~\ref{section:parametric}.  Theorem~\ref{thm:TCL} and Corollary~\ref{coro:true_clt} then hold as a consequence of~\cite[Theorem 3.11]{ClinetYoshida}, provided that one verifies two minor points of attention. One should firstly show that~\cite[Lemma 3.10]{ClinetYoshida} correctly extends to the locally stationary setting, and that the singularities induced by the existence of boundary coefficients in $\eta^*$ do fall within the scope of the preceding results. \\

The first point reduces to showing that the master Theorem for \textsc{m}-estimator (see for instance \cite[Theorem 5.7]{VanDenVaartAsymptoticStatistics}) does apply to the mixture limit functions found in Proposition~\ref{prop:ergodicity}, as we do in the proof of Proposition~\ref{prop:consistent}. The second point is straightforward: the constrained coordinates in $\eta$ live in some positive orthant and one needs not introduce a sequence of approximating cones. Under this very simple geometry, the array of results obtained as a consequence of Propositions~\ref{prop:moments} and~\ref{prop:ergodicity} are indeed sufficient conditions for Theorem~\ref{thm:TCL} in its full generality. See specifically that Proposition~\ref{prop:consistent}, and Lemmata 3.12 and 3.13 in~\cite{ClinetYoshida} respectively guarantee the sufficient conditions for Theorem~\ref{thm:TCL} given in Assumption $1$, and Assumptions $2^{2*}$ and $3^*$ of~\cite[p 695]{AndrewsDF}. 

\begin{proposition}\label{prop:consistent}
    Work under the setting of Section~\ref{section:parametric}, with, in particular, $g^*=g(x,\varpi^*)$ as in~\eqref{equ:well_param}. Grant Assumptions~\ref{ass:lipschitz} to~\ref{ass:integrability}. Then, the \textsc{mle} $\hat{\vartheta}_T=(\hat{\eta}_T,\hat{\varpi}_T)$ as in~\eqref{equ:parametric_mle} is a consistent estimator of $\vartheta^*=(\eta^*,\varpi^*)$.
\end{proposition}
\begin{proof}
    Re-arranging the terms of the log-likelihood~\eqref{equ:loglkl_def}, one finds that $\hat{\vartheta}_T$ maximises the functional\begin{equation*}\Lambda_T(\vartheta,\vartheta^*)
    =
    \sum_{k=1}^K \frac{1}{T} \int_0^T  \lambda^T_{k,t}(\vartheta^*) \Big\{ \log \frac{\lambda^T_{k,t}(\vartheta)}{\lambda^T_{k,t}(\vartheta^*)} - \Big( \frac{\lambda^T_{k,t}(\vartheta)}{\lambda^T_{k,t}(\vartheta^*)} -1  \Big)\Big\} \dif t
        +
        \frac{1}{T} \int_0^T   \log \frac{\lambda^T_{k,t}(\vartheta)}{\lambda^T_{k,t}(\vartheta^*)} \dif M^T_{k,t}.
    \end{equation*}
    The martingale term  $\mathcal{M}^T_t(\vartheta)=T^{-1} \sum_{k}^K \int_0^t \log \lambda^T_{k,s}(\vartheta) -  \log \lambda^T_{k,s}(\vartheta^*) \dif M^T_{k,s} $ vanishes as $T \to \infty$ as a consequence of Proposition~\ref{prop:moments}, according to
    \begin{equation*}
     \mathbb{E}\bigl[ 
    \lvert
    \mathcal{M}^T_T (\vartheta)\rvert^2
    \bigr]
    \leq 
    \frac{2}{T^2} \sum_{k=1}^K
    \int_0^T 
    \mathbb{E}\Bigl[ \bigl(\ln(\lambda^T_{k,s}(\vartheta))^2 +\ln(\lambda^T_{k,s}(\vartheta^*)) ^2 \bigr) \lambda^T_{k,s}(\vartheta^* )
    \Bigr]
    \dif s
    = \mathcal{O}(\frac{1}{T}),
    \end{equation*}
    where the bound holds uniformly in $\vartheta$ in view of Proposition~\ref{prop:moments}. Similarly, for any $\vartheta \in \Theta$, the martingale $(\partial_{\vartheta} \mathcal{M}^T_t(\vartheta))$ with coordinates $\partial_{\vartheta_i} \mathcal{M}^T_t(\vartheta)= T^{-1} \sum_{k=1}^K \int_0^t \partial_{\vartheta_i} \lambda^T_{k,t}(\vartheta) \lambda^T_{k,t}(\vartheta)^{-1} \dif M^T_{k,s} $ obeys
    \begin{equation*}
        \mathbb{E}\bigl[\lVert \partial_{\vartheta_i} \mathcal{M}^T_T(\vartheta)\rVert_2^2\bigr]
        \leq 
        \frac{C}{T^2}
        \sum_{k=1}^K
        \int_0^T
        \mathbb{E}\bigl[ \partial_{\vartheta} \lambda^T_{k,s}(\vartheta)^2 \lambda^T_{k,s}(\vartheta^*)\bigr]
        \dif s
        =
        \mathcal{O}(\frac{1}{T}),
    \end{equation*}
  for some $C>0$, where we have used the fact that, under Assumption~\ref{ass:no_inhibition}, $\lambda^T_{k,t}(\vartheta)$ is bounded from below for any $k = 1 \hdots K$, so as to apply Proposition~\ref{prop:moments}. Do note that the preceding results extend uniformly in $t$ via Doob's maximal inequality. They extend uniformly in $\vartheta$ too via the regularity argument in~\cite[Proof of Lemma 3.10 p. 1808]{ClinetYoshida}, which produces some positive constant $C(\Theta)>0$ such that
    \begin{equation*}
    \mathbb{E}\bigl[\sup_{\vartheta \in \Theta}
        \lvert 
        \mathcal{M}^T_T(\vartheta)
        \rvert^2
        \bigr]
        \leq 
        C(\Theta)
        \Bigl( 
        \int_{\Theta}
            \mathbb{E}[\lvert \mathcal{M}^T_T(\vartheta) \rvert^2 ]
        \dif \vartheta
        +
        \int_{\Theta}
            \mathbb{E}[
            \lVert\partial_{ \vartheta }\mathcal{M}^T_T(\vartheta) \rVert_2^2]
        \dif \vartheta
        \Bigr).
    \end{equation*}
    This proceeds from Assumption~\ref{ass:integrability} and Sobolev's inequality, in regard to which we direct again to the proof of Lemma 3.10 in~\cite[p. 1808]{ClinetYoshida} and the references therein (specifically, to \cite[Theorem 4.12, Part I, case
A, p. 85]{sobolev}). Then, the remaining term, hence $\frac{1}{T}\Lambda_T(\vartheta,\vartheta^*)$ itself, converges uniformly in probability as $T \to \infty$ towards
    \begin{equation*}
    V(\vartheta)=\sum_{k=1}^K\int_0^1
        \mathbb{E}\Bigl[ \lambda^{x,\infty}_{k,0}(\vartheta^*) \Bigl\{ \ln\Big(  \frac{\lambda^{x,\infty}_{k,0}(\vartheta)}{\lambda^{x,\infty}_{k,0}(\vartheta^*) } \Big) - \Big( \frac{\lambda^{x,\infty}_{k,0}(\vartheta)}{\lambda^{x,\infty}_{k,0}(\vartheta^*) }-1\Big) \Bigr\}\Bigr] \dif x,
    \end{equation*}
    as per Proposition~\ref{prop:ergodicity}. It is clear from the elementary inequality $\ln(x) - (x-1) \leq 0$  that the limit function $V$ admits a null global maximum on $\Theta$ at $\vartheta=\vartheta^*$. For the \textsc{m}-estimator master Theorem to apply, we must verify that it is a well-separated in the sense of~\cite[Theorem 5.7]{VanDenVaartAsymptoticStatistics}. Since $\vartheta$ lives in the compact subset $\Theta = \Gamma \times \Xi$, over which $V$ admits a global maximum, it suffices that $V$ be continuous (hence uniformly continuous) over $\Theta$. Note to this effect that $x \mapsto \ln(x)-(x-1)$ is Lipschitz-continuous over the range of values of the $\lambda_{k,0}^{x,\infty} (\lambda_{k,0}^{x,\infty})^{-1}$, with $k=1\hdots K$, under Assumption~\ref{ass:no_inhibition}. Hence one needs only the uniform equi-continuity in $L^1(\Prob(\vartheta^*))$-norm of the family of random functions $\{ \vartheta \mapsto \lambda_{k,0}^{x,\infty}(\vartheta) \lvert x \in [0,1] \}$. For any $\vartheta,  \vartheta' \in \Theta$ and any $k=1 \hdots K$ then,
    \begin{equation*}
       \mathbb{E}\left[ \frac{\lambda_{k,0}^{x,\infty}(\vartheta) -\lambda_{k,0}^{x,\infty}(\vartheta') }{\lambda^{x,\infty}_{k,0}(\vartheta^*)} \right]
       \leq 
       \frac{\mathbb{E}\big[\lambda_{k,0}^{x,\infty}(\vartheta) 
       -
       \lambda_{k,0}^{x,\infty}(\vartheta')\big] }{\inf_{x \in [0,1]} \Phi_k\big[\mu_k(x,\vartheta^*)\big]},
    \end{equation*}
    where, by $1$-Lipschitz-continuity of $\Phi_k$, for any $k=1\hdots K$,
    \begin{align*}
        \mathbb{E}\big[\lambda_{k,0}^{x,\infty}(\vartheta) -\lambda_{k,0}^{x,\infty}(\vartheta')\big]
        &\leq 
        \lvert \mu_k(x,\vartheta) - \mu_k(x,\vartheta') \rvert
        \\
        &+
        \lvert g(x,\varpi) - g(x,\varpi') \rvert
        \sum_{l=1}^K \Big( \int_0^{\infty} \sup_{ \vartheta \in \Theta} \varphi_{kl}(s,\vartheta) \dif s\Big) \mathbb{E}[\lambda_{l,0}^{x,\infty}(\vartheta^*)]
        \\
        &+
        \sup_{(x,\varpi) \in [0,1] \times \Xi} g(x,\varpi)
        \sum_{l=1}^K
         \Big( \int_0^{\infty}  \varphi_{kl}(s,\vartheta)  - \varphi_{kl}(s,\vartheta') \dif s\Big) \mathbb{E}[\lambda_{l,0}^{x,\infty}(\vartheta^*)].
    \end{align*}
    Assumption~\ref{ass:integrability} guarantees the integrability of the $\varphi_{kl}$, and Proposition~\ref{prop:moments} the boundedness of $\mathbb{E}[\lambda^{x,\infty}_0(\vartheta^*)]$. The three terms hereinabove can then be made arbitrarily small for any close enough $\vartheta$ and $\vartheta'$ on account of the uniform equi-continuity of $g$ and $\mu$, ending the proof.
\end{proof}
\begin{remark}
    The result of Proposition~\ref{prop:consistent} supersedes Assumption \textup{\textbf{[\textsc{a}4]}} in~\cite{ClinetYoshida}, which does not appear in our discussion as a consequence. See in particular that the limit function $V$ in the proof of Proposition~\ref{prop:consistent} enjoys a well-separated zero as its global maximum.
\end{remark}
As stated above, Corollary~\ref{coro:true_clt} now deduces from \cite{ClinetYoshida}. In order to provide sufficient context for the proof of Theorem~\ref{tm:LRT_for_g_constant_or_not}, we will nonetheless retrace here, at least formally, the key steps of the proof of the asymptotic normality of $\hat{\vartheta}_T=(\hat{\eta}_T,\hat{\varpi}_T)$ in the setting of Corollary~\ref{coro:true_clt}. Let us first note, as in~\cite{OgataMLE}, that the score at the true parameter $\partial_{\vartheta} \mathcal{L}_T(\vartheta^*)$ has a martingale structure. For any $T>0$,
\begin{equation*}
 \bigl(\frac{1}{\sqrt{T}}\partial_{\vartheta_i}\mathcal{L}_T(\vartheta^*)\bigr)_i
 =
 \Bigl(
   \frac{1}{\sqrt{T}}
   \sum_{k=1}^K 
    \int_0^T \frac{\partial_{\vartheta_i}\lambda^T_{k,s}(\vartheta^*)}{\lambda^T_{k,s}(\vartheta^*)} \dif M^T_{k,s} \Bigr)_i,
\end{equation*}
which is indeed the value at $t=T$ of the martingale $(\textswab{M}^T_t)$ with values in $\R^{d+p+1}$, which coordinates read
    $(\textswab{M}^T_{i,t})= (T^{-\nicefrac{1}{2}}\sum_{k=1}^K \int_0^t H^{(i)}_{k,s} \dif M^T_{k,s})$, where the $(H^{(i)}_{k,s})=  (\partial_{\vartheta_i}\lambda^T_{k,s}(\vartheta^*)\lambda^T_{k,s}(\vartheta^*)^{-1})$ are  predictable, and square integrable under Assumption~\ref{ass:no_inhibition} and Proposition~\ref{prop:moments}. The martingale in question verifies
\begin{equation*}
    \langle \textswab{M}^T_{i,\cdot}, \textswab{M}^T_{j,\cdot}  \rangle_t
    =
    \frac{1}{T}
   \sum_{k=1}^K 
    \int_0^t \frac{\partial_{\vartheta_i}\lambda^T_{k,s}(\vartheta^*)\partial_{\vartheta_j}\lambda^T_{k,s}(\vartheta^*)}{\lambda^T_{k,s}(\vartheta^*)} \dif s,
\end{equation*}
where we have used the fact that $\langle M^T_{k,\cdot}, M^T_{l,\cdot} \rangle = 0 $ when $k \neq l$. We intend to apply the martingale central limit Theorem (precisely, the Lindeberg-Feller variation in~\cite[Theorem VIII.3-22]{jacodshy}) and require two sufficient conditions. Firstly, in view of Proposition~\ref{prop:ergodicity},  for any $u \in [0,1]$, the martingale $(\textswab{M}^T_t)$ obeys
\begin{equation}\label{equ:Fisher_appears}
     \langle \textswab{M}^T_{i,\cdot}, \textswab{M}^T_{j,\cdot}  \rangle_{Tu}
     \to 
     \sum_{k=1}^K
     \int_0^u 
     \mathbb{E}\left[ 
     \frac{\partial_{\vartheta_i}\lambda^{x,\infty}_{k,0}(\vartheta^*)\partial_{\vartheta_j}\lambda^{x,\infty}_{k,0}(\vartheta^*)}{\lambda^{x,\infty}_{k,0}(\vartheta^*)}  \right] \dif x =
     \sum_{k=1}^K \int_0^u I^{(k)}_{ij}(\vartheta^*) \dif x.
\end{equation}
Secondly, for any $\epsilon>0$ and $i=1\hdots (d+p+1)$ 
\begin{align*}
    \mathbb{E} \Big[  \frac{1}{T} \int_0^T   \frac{ \bigl( \partial_{\vartheta_i}\lambda^T_{k,s}(\vartheta^*) \bigr)^4}{\bigl( \lambda^T_{k,s} (\vartheta^*) \bigr)^2} 
    &\mathbb{1}_{ \Big\{ \frac{\lvert  \partial_{\vartheta_i} \lambda^T_{k,s}(\vartheta^*) \rvert^2}{\lambda^T_{k,s}(\vartheta^*) }> \sqrt{T} \epsilon   \Big\}} \dif N_{k,t}^T
    \Big]
    \\
 &\leq  \sup_{t \in [0,T]} \mathbb{E} \Bigl[ \frac{\bigl( \partial_{\vartheta_i}\lambda^T_{k,s}(\vartheta^*)\bigl)^8}{\lambda^T_{k,s}(\vartheta^*)^2 }  \Bigr]^\frac{1}{2}
 \int_0^T
    \Prob \Bigl[   \frac{ \lvert \partial_{\vartheta_i}\lambda^T_{k,t}(\vartheta^*) \rvert^2}{\inf_x \Phi_k[\mu_k(x)]} > \sqrt{T} \epsilon  \Bigr]^\frac{1}{2}  \frac{\dif t}{T} 
    \\
    &= \mathcal{O}\big( T^{-\frac{1}{2}}\big),
\end{align*}
where we have used  Cauchy-Schwarz' inequality together with Proposition~\ref{prop:moments} in the first line, Markov's inequality and Assumption~\ref{ass:no_inhibition} in the last.  Clearly, the preceding domination together with the elementary bound
\begin{equation*}
     \frac{\partial_{\vartheta_i}\lambda^T_{k,s}(\vartheta^*)\partial_{\vartheta_j}\lambda^T_{k,s}(\vartheta^*)}{\lambda^T_{k,s}(\vartheta^*)}
     \leq 
     \frac{1}{2}\Bigl( 
       \frac{\partial_{\vartheta_i}\lambda^T_{k,s}(\vartheta^*)^2}{\lambda^T_{k,s}(\vartheta^*)}
       +
       \frac{\partial_{\vartheta_i}\lambda^T_{k,s}(\vartheta^*)^2}{\lambda^T_{k,s}(\vartheta^*)}
     \Bigr)
\end{equation*}
satisfies the Lindeberg condition of the martingale \textsc{clt}. The process $T^{-\nicefrac{1}{2}} (\textswab{M}^T_t)$ then converges in the Skorokhod topology -- hence in finite dimensional distributions -- towards a continuous Gaussian process. In particular,  at $t=T$, one recovers
\begin{equation}\label{equ:score_is_normal}
    \frac{1}{\sqrt{T}}\partial_{\vartheta}\mathcal{L}_T(\vartheta^*) \xrightarrow[T \to \infty]{\mathcal{L}(\Prob(\vartheta^*))}\mathcal{N}(0,I(\vartheta^*)),
\end{equation}
with $I(\vartheta^*)$ defined in~\eqref{equ:Fisher_appears}. Expanding the score and inverting the second order derivative,
\begin{equation}\label{equ:inversion}
    \sqrt{T}( \vartheta_T- \vartheta^* ) = I_T(\bar{\vartheta}_T )^{-1} \frac{1}{\sqrt{T}} \partial_{\vartheta} \mathcal{L}_T(\vartheta^*) + o_{\Prob(\vartheta^*)}(1),
\end{equation}
where $\bar{\vartheta}_T$ is on the line joining $\vartheta^*$ to $\hat{\vartheta}_T$, and $I_T(\vartheta)$ is the matrix function 
\begin{align*}
    I_T(\vartheta)
    =
    \sum_{k=1}^K
    \frac{1}{T}
    \int_0^T 
    \Bigl( 
    \frac{\partial^{\otimes 2} \lambda_{k,s}^T(\vartheta)}{\lambda^T_{k,s}(\vartheta)}
    &-
    \frac{\partial^{\otimes 2}_{\vartheta} \lambda^T_{k,s}(\vartheta)}{\lambda^T_{k,s}(\vartheta^*)}
    -\frac{(\partial_{\vartheta} \lambda^T_{k,s}(\vartheta))^{\otimes 2}}{\lambda^T_{k,s}(\vartheta)^2}
    \Bigr) \lambda^T_{k,s}(\vartheta^*) \dif s
    \\
    &+
    \frac{1}{T}
    \int_0^T
    \frac{(\partial_{\vartheta} \lambda^T_{k,s}(\vartheta))^{\otimes 2}}{\lambda^T_{k,s}(\vartheta)^2}
    - \frac{\partial^{\otimes 2}_{\vartheta} \lambda^T_{k,s}(\vartheta)}{\lambda^T_{k,s}(\vartheta)} \dif M^T_{k,s},
\end{align*}
where we have re-employed the notation $\partial^{\otimes 2}$ for the Hessian matrix, and written $x^{\otimes 2} = x^{\top}x$.  For any $\vartheta \in \Theta$, the random matrix function $I_T(\vartheta)$ defined hereinabove decomposes as the sum of a functional of $(\lambda^T_t(\vartheta^*),\lambda^T_t(\vartheta),\partial_{\vartheta}\lambda^T_t(\vartheta),\partial^2_{\vartheta} \lambda^T_t(\vartheta))$ with a martingale.  Proceeding exactly as we have done in the proof of Proposition~\ref{prop:consistent}, that is, by a successive application of Doob's maximal inequality, Proposition~\ref{prop:moments}, and the Sobolev inequality of~\cite[proof of Lemma 3.10 p. 1808]{ClinetYoshida}, one finds the martingale term vanishes uniformly in $t$ and $\vartheta$ as $T \to \infty$. The first term is then subject to Proposition~\ref{prop:ergodicity}, whence it converge uniformly in $\vartheta$ towards 
\begin{equation*}
    I(\vartheta) = \sum_{k=1}^K \int_0^1  \Bigl(
    \frac{\partial^{\otimes 2} \lambda_{k,s}^{x,\infty}(\vartheta)}{\lambda^{x,\infty}_{k,s}(\vartheta)}
    - \frac{(\partial_{\vartheta} \lambda^{x,\infty}_{k,s}(\vartheta))^{\otimes 2}}{\lambda^{x,\infty}_{k,s}(\vartheta)^2}
    -
    \frac{\partial^{\otimes 2}_{\vartheta} \lambda^{x,\infty}_{k,s}(\vartheta)}{\lambda^{x,\infty}_{k,s}(\vartheta^*)}
    \Bigr) \lambda^{x,\infty}_{k,s}(\vartheta^*) \dif s
\end{equation*}
 in probability under $\Prob(\vartheta^*)=\Prob(\eta^*,g(\cdot,\varpi^*))$. See that the preceding notation coincides with~\eqref{equ:asymptotic_information} at $\vartheta=\vartheta^*$. Together with Proposition~\ref{prop:consistent}, one retrieves $I_T(\bar{\vartheta}_T) \to I(\vartheta^*)$ in probability as $T \to \infty$. Note also that this justifies the inversion in~\eqref{equ:inversion} under Assumption~\ref{ass:fisher_is_non_degenerate}, by virtue of the invertible matrices being an open set of $\mathcal{M}_{(p+d+1)}(\R)$. Slutsky's Lemma allows one to conclude, with
\begin{equation}\label{equ:tcl_is_here}
    \sqrt{T}(\hat{\vartheta}_T- \vartheta^*)
    \to \mathcal{N}(0,I(\vartheta^*)^{-1}),
\end{equation}
which is the desired convergence.

\subsection{Proof of Theorem~\ref{tm:LRT_for_g_constant_or_not}} Work under the parametric setting of Section~\ref{section:parametric}, still, and retain the notations $\vartheta^*=(\eta^*,\varpi^*)$ and $\Prob(\vartheta^*)$ for the true parameter and associated probability measure. The composite hypothesis 
\begin{equation*}
    \Xi_\mathfrak{d}^0= \bigl\{ (\varpi_i) \in \Xi^d
    \vert 
\varpi_0 = \cdots = \varpi_{\mathfrak{d}-1}  
    \bigr\}
\end{equation*}
rephrases as a simple hypothesis by introducing a nuisance parameter $\nu>0$, with
\begin{align*}
    \varpi_{\mathfrak{d}}=\nu, ~ \:
    \varpi_{\mathfrak{d}-1} = \nu + \delta_{d-1}, ~ \:
    \hdots, ~ \:
    \varpi_{0} = \nu + \delta_0,
\end{align*}
where $(\delta_i) \in \R^{\mathfrak{d}}$, in such manner that 
\begin{equation*}
    \Xi_\mathfrak{d}^0 \equiv \bigl\{ 
    \delta_0 = \cdots = \delta_{\mathfrak{d}-1}=0  
    \bigr\}.
\end{equation*}
 The model is then equivalently parametrized in terms of $(\eta,\nu,\delta)$ with the two parameter spaces indeed mapping to the same model, and their correspondence given by $ \varpi  = G(\nu,\delta^\top)^\top$, where $G \in \mathcal{M}_{\mathfrak{d}+1}(\R)$ is the invertible matrix
\begin{equation*}
    G = \begin{bmatrix}
       \boldsymbol{Id}_{\mathfrak{d}} &\begin{matrix} 1 \\
       \vdots \\
           1
       \end{matrix} \\
       \begin{matrix}
           0 & \hdots &0 
       \end{matrix}
       &
       \begin{matrix}
           1
       \end{matrix}
    \end{bmatrix}.
\end{equation*}
 The score and the Fisher Information of the two equivalent models then coincide up to some invertible multiplicative factors $E \in \mathcal{M}_{p+\mathfrak{d}+1}(\R)$ of the form
\begin{equation*}
    E = \begin{bmatrix}
       \boldsymbol{Id}_p&0 \\
        0 &   G^{-1} 
    \end{bmatrix},
\end{equation*}
which do not intervene in the limit distribution of $\Lambda^{\mathfrak{d}}_T$. 
Up to a linear re-parametrization, one may therefore suppose without loss of generality that  $(\vartheta_{p+1},\hdots,\vartheta_{p+\mathfrak{d}+1})$ lives in a compact set of $\R^{\mathfrak{d}}$, which non empty interior contains $0$, and the result of Theorem~\ref{tm:LRT_for_g_constant_or_not} proven for null hypotheses of the type  $\vartheta_{p+1} = \cdots \vartheta_{p+\mathfrak{d}-1}=0$ instead. The resulting setting is similar in every point to the standard case in~\cite[section 16.2]{VanDenVaartAsymptoticStatistics}, and the same arguments apply. Namely, introducing the sparse \textsc{mle},
\begin{equation*}
    \hat{\vartheta}^0_T 
    = \argmax_{\{ \vartheta_{p+1}= \cdots = \vartheta_{p+\mathfrak{d}+1} =0 \}} \mathcal{L}_T(\vartheta).
\end{equation*}
One has at any $T>0$ 
\begin{align}
    \mathcal{L}_T(\hat{\vartheta}_T)
    -
     \mathcal{L}_T(\hat{\vartheta}_T^0)
    &=  
    \sqrt{T}  (
    \hat{\vartheta}_T
    -
   \hat{\vartheta}_T^0)^{\top}
    \frac{1}{T}
    \partial^{\otimes 2}_{\vartheta} \mathcal{L}_T(\vartheta^*) \sqrt{T}  (
    \hat{\vartheta}_T
    -
   \hat{\vartheta}_T^0)
    +o_{\Prob(\vartheta^*)}(1),
    \label{equ:vanderwaartexpansion}
\end{align}
where we have cancelled the first order term on account of the gradient of $\mathcal{L}_T$ vanishing at $\hat{\vartheta}_T$ under the Assumptions of Corollary~\ref{coro:true_clt}.  Furthermore, expanding the score,
\begin{equation*}
    \sqrt{T}(\hat{\vartheta}_T - \vartheta^*) 
    =
     I(\vartheta^*)^{-1}  
    \frac{1}{\sqrt{T}}\partial_{\vartheta} 
    \mathcal{L}_T(\vartheta^*) 
    + o_{\Prob(\vartheta^*)}(1),
\end{equation*}
and, likewise, 
\begin{equation*}
    \sqrt{T} (
    \hat{\vartheta}^0_T
    -
   \vartheta^*)_{\leq p}
   =
    (I(\vartheta^*)_{\leq p, \leq p})^{-1}
    (\frac{1}{\sqrt{T}} \partial_{\vartheta}
    \mathcal{L}_T(\vartheta^*) )_{\leq p}
    +
    o_{\Prob(\vartheta^*)}(1),
\end{equation*}
where, for any $A \in \mathcal{M}_{p+\mathfrak{d}+1}(\R)$, $A_{\leq p,\leq p}$ denotes the $p$-th first principal submatrix of $A$, and the same notation readily extends to vectors of $\mathbb{R}^{p+\mathfrak{d}+1}$. With the Schur formula
\begin{equation*}
    (I(\vartheta^*)^{-1})_{\leq p \leq p} = (I(\vartheta^*)_{\leq p \leq p})^{-1} I(\vartheta^*)_{\leq p > p}(I(\vartheta^*)^{-1})_{>p >p}I(\vartheta^*)_{\leq p > p}^\top (I(\vartheta^*)_{\leq p \leq p})^{-1} + (I(\vartheta^*)_{\leq p \leq p})^{-1},
\end{equation*}
and elementary matrix manipulation, one finds that the two \textsc{mle}s are related via
\begin{equation*}
    \sqrt{T}(\hat{\vartheta}^0_T-\hat{\vartheta}_T )=( I(\vartheta^*)_{\leq p \leq p})^{-1}  I(\vartheta^*)_{\leq p > p} \sqrt{T}\hat{\vartheta}_{ >p} + o_{\Prob(\vartheta^*)}(1),
\end{equation*} and that
expression~\eqref{equ:vanderwaartexpansion} is asymptotically equivalent to 
\begin{equation*} 
   ( \sqrt{T}\hat{\vartheta}_T )_{>p}^{\top}\Bigl( \bigl(I({\vartheta^*}) ^{-1}\bigr)_{>p,>p}\Bigr)^{-1}
     ( \sqrt{T}\hat{\vartheta}_T)_{>p},
\end{equation*}
which converges weakly towards a $\chi^2(k)$ law where $k=p+\mathfrak{d}+1-(p+1)=\mathfrak{d}$ under any $\Prob(\vartheta^*)$ with $\vartheta^*$ in the null hypothesis $\vartheta^*_{>p} = (0 \hdots 0)^\top$ as per Corollary~\ref{coro:true_clt}. This ends our proof of Theorem~\ref{tm:LRT_for_g_constant_or_not}. We now continue towards the proof of our technical Propositions, which are set in the general setting hence we retire the parametric notation $\lambda^T_t(\vartheta)=\lambda^T_t(\eta,g(\cdot,\varpi))$ in favour of the general one. 
\subsection{Proof of Proposition~\ref{prop:moments}}\label{section:proof_of_moments}
Define for any $\vartheta=(\eta,\varpi) \in \Theta$ the bounding kernels
\begin{align}
\bar{\varphi}(t,\vartheta) 
&=
\sup_{x \in [0,1]} g(x,\varpi)\bar{\varphi}(t,\eta)
\nonumber
\\
\bar{\varphi}^*(t) &= \sup_{x \in [0,1]}g^*(x) \varphi(t,\eta^*)
\label{equ:bar_phi}
\end{align}
and the associated resolvents
\begin{align*}
\bar{\Psi}^*
    &=
   \sum_{k=1}^{\infty} (\bar{\varphi}^*)^{\star k}\\
   \bar{\Psi}^\Theta
   &=
  (\sup_{\vartheta \in \Theta} \bar{\varphi}_{kl}(\cdot,\vartheta))
  +
  (\sup_{\vartheta \in \Theta} \bar{\varphi}_{kl}(\cdot,\vartheta) ) \star \bar{\Psi}^*.
\end{align*}
In particular, wherever the parametric specification~\eqref{equ:well_param} of Section~\ref{section:parametric} applies, one has the simplification $\bar{\varphi}^*=\bar{\varphi}(\cdot,\vartheta^*)$. The well-posedness of $\bar{\Psi}^*$ and $\bar{\Psi}^\Theta$ results from the stability Assumption~\ref{ass:stability} and Young's convolution inequality. Additional details can for instance be found in~\cite[Lemma 3]{bacrylimit} or~\cite{gripenberg}.
\begin{Lemma}\label{Lemma:jaisson_like}
For any $T>0$, any $t \in [0,T]$ and any $\vartheta =(\eta,\varpi)\in \Theta$,
    \begin{equation*}
        \lambda_t^T(\vartheta)
        \leq 
        \Phi[\mu\big( \frac{t}{T},\eta \big)]
        +
        \int_0^t \bar{\Psi}^{\Theta}(t-s)
        \Phi[\mu( \frac{s}{T},\eta^* \big)]
        \dif s
        +
        \int_0^t \bar{\Psi}^{\Theta}(t-s)
        \dif M^T_s
    \end{equation*}
    coordinate-wise. 
\end{Lemma}
\begin{proof}
    The proof is similar to~\cite{JaissonRosenbaum}. By Lipschitz-continuity of the $\Phi_k$, the true intensity $(\lambda^T_t(\vartheta^*,g^*))$ verifies for any $0 \leq t \leq T$ and $k=1\hdots K$
    \begin{align*}
        \lambda^T_{k,t}(\eta^*,g^*)
        & \leq 
        \Phi_k\big[\mu_k(\frac{t}{T},\eta^*)\big]
        +
        \sum_{l=1}^K \int_0^t \bar{\varphi}_{kl}^*(t-s)
         \dif N^T_{l,s}.
    \end{align*}
    Separating the martingale and finite variation part in $\dif N^T_s$, this yields the renewal inequation
    \begin{equation}\label{equ:first_appearance_of_renewal}
        \lambda^T_t(\eta^*,g^*)
        \leq 
        \Phi\big[\mu(\frac{t}{T},\eta^*)\big]
        +
        \int_0^t \bar{\varphi}^*(t-s)
        \lambda_t^T(\eta^*,g^*) \dif s
        +
        \int_0^t\bar{\varphi}^*(t-s)
         \dif M^T_s,
    \end{equation}
    which holds coordinate-wise and resolves into
    \begin{equation}\label{equ:jaisson_form}
        \lambda^T_t(\eta^*,g^*)
        \leq 
         \Phi\bigl[\mu(\frac{t}{T},\eta^*\big)\bigr]
        +
        \int_0^t \bar{\Psi}^*(t-s)
        \mu( \frac{s}{T},\eta^* \big)
        \dif s
        +
        \int_0^t \bar{\Psi}^*(t-s)
        \dif M^T_s,
    \end{equation}
    where we have used the fact that $\bar{\varphi}^*+\bar{\varphi}^* \star \bar{\Psi}^*=\bar{\Psi}^*$.  The general case $\vartheta \in \Theta$ follows by writing
    \begin{align*}
        \lambda^T_t(\eta,g(\cdot,\varpi))
        &\leq 
       \Phi[ \mu(\frac{t}{T},\vartheta)]        +
        \int_0^t
            \sup_{\eta \in \Gamma} \bar{\varphi}(t-s,\eta)
        \lambda^T_s(\eta^*,g^*) 
        \dif s
        +
        \int_0^t 
             \sup_{\eta \in \Gamma} \bar{\varphi}(t-s,\eta)
        \dif M^T_s,
    \end{align*}
    where the suprema are understood in the coordinate-wise sense. It then suffices to re-insert the inequality obtained at $(\eta,g)=(\eta^*,g^*)$ in the expression hereinabove, and apply Fubini's Theorem. 
\end{proof}

\begin{Lemma}\label{lemma:recursive_moments}
    Work under Assumptions~\ref{ass:lipschitz} and~\ref{ass:stability}. Let $p \in \mathbb{N}$ and $q=2^p$. If $\sup_{\vartheta} \varphi(\cdot,\vartheta)  \in L^{q}[0,\infty)$, then,
    \begin{equation*}
    \sup_{T>0} \sup_{t \in [0,T]} 
        \mathbb{E}\big[
        \lVert 
        \sup_{\vartheta \in \Theta}
        \lambda^T_t
        (\vartheta)
        \rVert^{q}_{q}
        \bigr]
        < \infty.
    \end{equation*}
\end{Lemma}
\begin{proof}
    In view of Lemma~\ref{Lemma:jaisson_like}, a sufficient condition is that, for any $\psi=(\psi_{kl}) \in L^1[0,\infty) \cap L^{2^p}[0,\infty)$     \begin{equation*}
        \mathbb{E}\Big[  \big\lVert \int_0^t \psi(t-s) \dif M^T_s \big\rVert^{2^p}_{2^p} \Big]
       =
       \sum_{k=1}^K
       \sum_{l=1}^K
       \mathbb{E}\Bigl[
       \bigl\lvert \int_0^t \psi_{kl} (t-s) \dif M^T_{l,s} \bigr\rvert^{2^p}     
       \Bigr] <\infty.
    \end{equation*}
     Recalling that $\bar{\Psi}^\Theta$ enjoys the same integrability properties as $\sup_{\vartheta} \varphi(\cdot,\eta)$,  the Lemma  indeed follows from the case $\psi= \bar{\Psi}^\Theta$. Let  $\psi \in L^1[0,\infty)$. We proceed by induction. At $p=0$, see that  
     \begin{equation*}
         \sum_{k=1}^K\sum_{l=1}^K \mathbb{E}[ \lvert \int_0^t \psi_{kl}(t-s) \dif M^T_{l,s} \rvert ]  
         \leq 
         2 \sum_{k=1}^K \sum_{l=1}^K \int_0^t \lvert \psi_{kl}(t-s) \rvert  \mathbb{E}[\lambda^T_{l,s}(\vartheta^*,g^*)] \dif s. 
     \end{equation*}
      By Lemma~\ref{Lemma:jaisson_like}, $(\lambda^T_{t}(\vartheta^*,g^*))$ obeys the uniform bound
      \begin{equation*}
         \sum_{k=1}^K 
         \mathbb{E}[ \lambda^T_{k,t}(\eta^*,g^*)]
         \leq 
        \sum_{k=1}^K  \sup_{x \in [0,1]}\Phi_k[\mu_k(x,\eta^*)]
        +
        \sum_{k=1}^K  \sum_{l=1}^K \Bigl( \int_0^{\infty} \bar{\Psi}^*_{kl}(s) \dif s \Bigr) \sup_{x \in [0,1]}\Phi_l[\mu_l(x,\eta^*)],
      \end{equation*}
      thus the initial case $p=0$. Now, grant the property at some $p \geq 1$, and suppose that $\psi \in L^{q} [0,\infty)$ where $q=2^{p+1}$. For any $T>0$ and $ t \in [0,T]$, the process $(\int_0^x \psi(t-s) \dif M^T_s)_{x \geq 0}$ is a martingale and, by the  Burkholder-Davis-Gundy (\textsc{bdg}) inequality (see~\cite{kühn2023maximalinequalitiesapplications}),
    \begin{equation*}
        \sum_{k=1}^K  \sum_{l=1}^K  \mathbb{E}\Big[\big\lvert\int_0^t \psi_ {kl}(t-s) \dif M^T_{l,s} \big\rvert^{2^{p+1}} \Big]
        \leq 
       C_p
        \sum_{k=1}^K 
         \sum_{l=1}^K 
       \mathbb{E}\Big[\big\lvert\int_0^t \psi_{kl}(t-s)^2 \dif N^T_{l,s} \big\rvert^{2^{p}} \Big]
    \end{equation*}
    for some $C_p >0$. Separating the finite variation and martingale terms in the random measure $\dif N^T_s$ the preceding estimate is itself bounded by
    \begin{equation*}
        C'_p
        \mathbb{E}\Big[
         \sum_{k=1}^K \big\lvert
         \int_0^t
            \psi_{kl}(t-s)^2\lambda^T_{l,s} (\vartheta^*,g^*)
         \dif s
        \big\rvert^{2^{p}} 
        \Big]
        +
        C'_p
        \mathbb{E}\Big[
         \sum_{k=1}^K \big\lvert
         \int_0^t 
            \psi_{kl}(t-s)^2 
        \dif M^T_{l,s} \big\rvert^{2^{p}},
        \Big]
    \end{equation*}
    where $C'_p=2^{2^ p-1} C_p$. The second term hereinabove is finite as per our induction hypothesis, since we have supposed that $\psi \in L^{2^{p+1}}[0,\infty)$, so that $\psi_{kl}^{2} \in L^{2^p}[0,\infty)$ for any $k,l= 1\hdots K$. Up to some multiplicative constant, the first term is bounded by
    \begin{equation*}
        \sum_{k=1}^K \sum_{k=1}^K \Big (\int_0^{\infty} \psi_{kl}^2(s) \dif s \Big)^{2^p}
        \sup_{T>0} \sup_ {t \in [0,T]}\mathbb{E}\Big[ \lvert\lambda^T_t(\eta^*,g^*) \rvert^{2^p}\Big],
    \end{equation*}
    where we have used Jensen's inequality with the normalization $(\int_0^{t}  \psi_{kl} ^2(s) \dif s )$ within each summand. As $\psi \in L^1[0,\infty) \cap L^q[0,\infty)$ with $q\geq2$, Riesz' interpolation Theorem ensures that $\psi \in L^2[0,\infty)$ and the multiplicative constant in our last bound is indeed finite. Lastly, $\mathbb{E}\big[ \lvert\lambda^T_t(\vartheta^*,g^*) \rvert^{2^p}]$ is uniformly bounded under our induction hypothesis as a consequence of Lemma~\ref{Lemma:jaisson_like}, ending the proof.
\end{proof}
\begin{remark}\label{remark:extend_to_derivatives}
    Extending Lemma~\ref{lemma:recursive_moments} to the derivatives of $\lambda_t(\vartheta)$ according to $\vartheta$ is straightforward. See to this end that, for any $k=1\hdots K$, $\Phi_k$ is smooth, strictly convex and Lipschitz-continuous. Its first derivative is then bounded, and its higher derivatives go to $0$ away from $0$, hence they are bounded too. As a result, the bounding of the moments of $\partial_{\vartheta} \lambda^T_t$ reduces to bounding
\begin{equation*}
\mathbb{E}\Big[ \big\lVert  
    \partial_{\vartheta} \mu_k\big( \frac{t}{T},\vartheta \big) 
    +
    \sum_{l=1}\int_0^t \partial_{\vartheta,\varpi}  \Big( g\big( \frac{t}{T}, \varpi \big)\varphi_{kl}(t-s,\vartheta) \Big)   \dif N^T_s\big\rVert^q_q
    \Big].
\end{equation*}
This indeed follows from the exact same arguments as for  $\lambda^T_t(\vartheta)$ itself, provided that the derivatives of $\mu,g$ and $\varphi$ satisfy the same regularity and integrability properties as  $\mu,g$ and $\varphi$ themselves, which Assumption~\ref{ass:integrability} ensures.
\end{remark}

Under Assumption~\ref{ass:integrability}, Proposition~\ref{prop:moments} follows readily from Lemma~\ref{lemma:recursive_moments} and Remark~\ref{remark:extend_to_derivatives}. We may proceed to the proof of Proposition~\ref{prop:ergodicity}.

\subsection{Uniform convergence of the stationary Hawkes processes}\label{section:proof_of_ergodicty}
Recall that we have defined in Section~\ref{section:introduction} a collection indexed on $[0,1]$ of stationary Hawkes processes  $(\boldsymbol{N}^{x,\infty}_t)$ with null initial condition, all imbedded into the same Poisson base as $(\boldsymbol{N}^T_t)$, each with respective intensity
\begin{equation*}
    (\lambda^{x,\infty}_{k,t}(\eta^*,g^*))
    =
   \big( \Phi_k \big[ 
    \mu_k(x,\eta^*)
    +
    \sum_{l=1}^K
    \int_{-\infty}^{t}
    g^*(x)
    \varphi_{kl}(t-s,\eta^*)
    \dif N_{l,s}^{x,\infty}
    \big] \big),
    \hspace{0.1cm}
    k=1\hdots K.
\end{equation*}
We reproduce here, mainly from~\cite{bremaud1996stability}, some known properties of the $(\boldsymbol{N}^{x,\infty}_t)$. We denote by $(\textfrak{M},B(\textfrak{M}))$ the locally finite Borel measures over $\R$ equipped with their vague topology. Recall also from Subsection~\ref{section:organisation} the notation $C_p(\R^K)$ and $C_{\uparrow}(\R^K)$ as borrowed from~\cite{ClinetYoshida}, and write $\theta_x \colon \textfrak{M} \mapsto \textfrak{M}$ for the shift operator over $\textfrak{M}$ defined by $\theta_x  m(A) = m(A-x)$ for any $A \in B(\R)$.

\begin{Lemma}\label{Lemma:stationary_ergodicity}
    Under Assumptions~\ref{ass:lipschitz} to~\ref{ass:integrability}, for any $x \in [0,1]$, any $\eta \in \Gamma$ and any $g \in C[0,1]$, $(\lambda^{x,\infty}_t(\eta,g))$ is $C_p(\R^K)$-ergodic in the sense that, for any $\psi \in C_p(\R^K)$,
    \begin{equation*}
        \frac{1}{T} \int_0^T \psi(\lambda_s^{x,\infty}(\eta,g)) \dif s
        \to 
        \mathbb{E}[\psi(\lambda^{x,\infty}_0(\eta,g)]
    \end{equation*}
    in probability under $\Prob(\eta^*,g^*)$ as $T \to \infty$.
\end{Lemma}
\begin{proof}
    As stated in Section~\ref{section:introduction}, we work with the construction of~\cite{bremaud1996stability}, wherein the stationary processes $(\boldsymbol{N}^{x,\infty}_t)$ are defined as limits of an iterative thinning scheme over the jumps of $(\pi_k)$, whence the random measures $N^{x,\infty}_k$ inherit the ergodicity and mixing properties (in the sense of \cite[Definition 12.3.I]{DVJ}). We refer specifically to the comments p.1573 and the proof of Theorem 7 in~\cite{bremaud1996stability} for technical details. In particular, for any $B(\textfrak{M})$-measurable function $f \colon \textfrak{M} \mapsto \R$ such that $\mathbb{E}[f(N^{x,\infty})]<\infty$,
    \begin{equation*}
        \frac{1}{T}\int_0^T
        f(\theta_tN^{x,\infty})\dif t
        \to \mathbb{E}[f(N^{x,\infty})],
        \textup{ as}
        \hspace{0.1cm}
        T \to \infty,
    \end{equation*}
    where the convergence occurs in probability. Since $\varphi(\cdot,\eta)$ may be expressed as a pointwise limit of compactly supported functions under Assumption~\ref{ass:continuity_in_time}, the function $F$ defined by
\begin{align*}
    \lambda^{x,\infty}_{k,u} (\eta,g)&= 
    F(\theta_u N_1^{x,\infty}, \hdots, \theta_u N_K^{x,\infty} )
    =
    \Big( 
    \Phi_k \Big[
    \mu_k (x, \eta)
    + 
    g(x)
    \sum_{l=1}^K
    \int_{\R} \varphi_{kl}(-s,\eta) \dif ( \theta_u N_l^{x,\infty})_s
    \Big] 
    \Big)_k
\end{align*}
is measurable from $(\textfrak{M},B(\textfrak{M}))$ to $(\R,B(\R))$ (see for instance~\cite[Example 7.1.4]{Bremaudbook} for a comparable argument). Furthermore, for any $k =1 \cdots k$, $F(\theta_u N_k^{x,\infty})$ admits a finite a first moment by consequence of Proposition~\ref{prop:moments}. This extends to moments of any order under Assumption~\ref{ass:integrability}, hence the same applies to any $f(F(N^{x,\infty}_k))$ where $f \in C_p(\R)$, and the Lemma follows.
\end{proof}
\begin{remark}\label{remark:extend_again_to_derivatives}
    Extending Lemma~\ref{Lemma:stationary_ergodicity} to derivatives of $(\lambda^{x,\infty}_t(\vartheta))=(\lambda^{x,\infty}_t(\eta,g(\cdot,\varpi)))$ according to $\vartheta=(\eta,\varpi)$ is again a purely procedural matter. It suffices that the $\partial_{\vartheta} \lambda^{x,\infty}(\vartheta)$ remain measurable functions of the random measures $N^{x,\infty}_k$ as they do under Assumption~\ref{ass:differentiability}. The existence of a finite first moment is comparably straightforward as discussed in remark~\ref{remark:extend_to_derivatives}.
\end{remark}

In view of Lemma~\ref{Lemma:stationary_ergodicity} and remark~\ref{remark:extend_again_to_derivatives}, we have obtained the following.

\begin{Lemma}\label{lemma:stationary_ergodicity_derivative}
    Under Assumptions~\ref{ass:lipschitz} to~\ref{ass:integrability}, for any $x \in [0,1]$, any $\vartheta \in \Theta$, and any $\psi \in C_p(\R^K \times \R^K \times \R^{K (d+p+1)}, \R^{K (d+p+1)^2})$,
    \begin{equation*}
    \frac{1}{T}
        \int_0^T
        \psi(\lambda^{x,\infty}_t(\eta^*,g^*),\lambda^{x,\infty}_{t}(\vartheta),\partial_{\vartheta} \lambda^{x,\infty}_{t}(\vartheta),
        \partial_{\vartheta} ^{\otimes 2}
        \lambda^{x,\infty}_{t}(\vartheta))
        \dif 
    \end{equation*}
   converges in probability under $\Prob(\eta^*,g^*)$ towards 
   \begin{equation*}
          \mathbb{E} \big[ 
       \psi(\lambda^{s,\infty}_0(\eta^*,g^*),\lambda^{x,\infty}_{0}(\vartheta),\partial_{\vartheta}  \lambda^{x,\infty}_{0}(\vartheta),
      \partial_{\vartheta} ^{\otimes 2}
        \lambda^{x,\infty}_{k,0}(\vartheta))
        \big]
   \end{equation*}
   as $T \to \infty$.
\end{Lemma}

\begin{remark}
    Lemmata~\ref{Lemma:stationary_ergodicity} and~\ref{lemma:stationary_ergodicity_derivative} may alternatively be obtained from the mixing of the random measures $N^{x,\infty}$ and Davydov's covariance inequality (see~\cite[equ. 1.12.b p. 6]{Rio}) together with Proposition~\ref{prop:moments}. A comparable line of reasoning is for instance followed in~\cite{kwan2023asymptotic}, where an argument specific to the cluster representation of the case $\Phi_k(x)=x$ is developed.
\end{remark}
In the sequel, we actually require that Lemma~\ref{lemma:stationary_ergodicity_derivative} holds in $L^1(\Prob(\vartheta^*,g^*))$. Hence Lemma~\ref{lemma:actually_L1}.
\begin{Lemma}\label{lemma:actually_L1}
    Under Assumption~\ref{ass:lipschitz} to~\ref{ass:integrability}, the convergences of Lemmata~\ref{Lemma:stationary_ergodicity} and~\ref{lemma:stationary_ergodicity_derivative} hold in $L^1(\Prob(\vartheta^*,g^*))$. 
\end{Lemma}
\begin{proof}
    It suffices to prove that, for any $\psi \in C_p(\R^K)$, the collection indexed on $T$ of random variables $\frac{1}{T} \int_0^T \psi(\lambda_t(\eta,g)) \dif t$ is uniformly integrable. However, this is straightforward as for any $\psi \in C_p(\R^K)$, one has some $C>0$ and $p >0$ such that $\lvert \Psi(x) \rvert < C(1+\lVert x \rVert^p_p) $, hence for any $\vartheta \in \Theta$ and $g \in C[0,1]$,
    \begin{equation*}
        \lvert \psi(\lambda^{x,\infty}_0(\eta,g)) \rvert
        \leq
        C(1+\lVert \lambda^{x,\infty}_0(\eta,g) \rVert^p_p). 
    \end{equation*}
    Since the $(\boldsymbol{N}^{x,\infty}_t)$ are all embedded into the same base $(\pi_k)$, the random variable $\lambda^{x,\infty}_0(\eta,g) $ is dominated term-by-term by the intensity of the stationary Hawkes process with constant baseline $(\sup_{x \in [0,1]}\mu_k(x,\eta))$ and kernel $t \mapsto \sup_{x \in [0,1]} g(x) \varphi(t,\eta)$, which admits bounded moments of any order as per Proposition~\ref{prop:moments}. The uniform integrability follows and the case of higher derivatives is similar in every point. 
\end{proof}

We have recalled in Lemmata~\ref{Lemma:stationary_ergodicity} to~\ref{lemma:actually_L1} the ergodicity of the $(N^{x,\infty}_{k,t})$ pointwise in $x$. It should hold uniformly in $x$ for the approximation scheme of Section~\ref{section:kcd} to converge, hence we require Lemma~\ref{lemma:continuity_in_dx}.

\begin{Lemma}\label{lemma:continuity_in_dx}Under Assumptions~\ref{ass:lipschitz} to~\ref{ass:continuity}, for any $\eta \in \Gamma$, $g \in C[0,1]$, and any $q \in \{1,2 \}$,
    \begin{equation*}
     \sup_{\lvert x-y\rvert < \delta }
        \mathbb{E}\Bigl[
        \big\lVert  \lambda^{x,\infty}
        _0(\eta,g)
        -
        \lambda^{y,\infty}
        _0(\eta,g)
        \big\rVert_q^q
        \Bigr]
        \leq 
        C_p
        \big( 
        \omega(g^*, \delta )^q
        +
        \omega(g, \delta )^q
        +
        \omega(\mu(\cdot,\eta), \delta  )^q
        +
        \omega(\mu(\cdot,\eta^*), \delta  )^q        
        \big)
    \end{equation*}
    where, for any $f \colon \R \mapsto \R^d$, $\omega$ denotes the modulus of continuity $\omega(f,d)=\inf_{\lvert x-y\rvert \leq d} \lVert f(x)-f(x)\rVert_1$. 
\end{Lemma}
\begin{proof} For any $\eta \in \Gamma$, $g \in C[0,1]$, $k=1 \hdots K$, and any $0 \leq t \leq T$,
    \begin{align}
        \lvert 
            \lambda^{x,\infty}_{k,t}(\eta,g)
            -
            \lambda^{y,\infty}_{k,t}(\eta,g)
        \rvert
        &\leq 
        \lvert \mu_k(x,\eta) - \mu_k(y,\eta) \rvert
        \nonumber
        \\
        &+
         \lvert g(x) - g(y) \rvert
       \sum_{l=1}^K
       \int_{- \infty}^t
       \varphi_{kl}(t-s,\eta) \dif N^{x,\infty}_{l,s}
       \label{equ:modulus}
       \\
       &+  
       \sum_{l=1}^K
       \int_{- \infty}^t
       g(y)\varphi_{kl}(t-s,\eta) \dif \lvert  N^{x,\infty}- N^{y,\infty} \rvert_{l,s},
       \nonumber
    \end{align}
     and the random measure $\lvert  N^{x,\infty}- N^{y,\infty} \rvert_{l,s}$ counts jumps occurring in either one of $N_l^{x,\infty}$ or $N_l^{y,\infty}$, but not the other. Since the $(N_t^{x,\infty})$ are all embedded into the same base $(\pi_k)$, the process
    \begin{equation*}
        (\tilde{M}^{x,y}_{k,t})
        =
        \bigl( \int_0^t 
         \dif \lvert  N^{x,\infty}- N^{y,\infty} \rvert_{k,s}
        -
        \int_0^t
        \lvert \lambda^{x,\infty}_{k,s}(\eta^*,g^*)
        -
        \lambda^{x,\infty}_{k,s}(\eta^*,g^*)]
        \dif s
        \bigr)
    \end{equation*}
    is a martingale, see again~\cite[page 1576]{bremaud1996stability}. Decomposing the random measure $\lvert N^{x,\infty}- N^{y,\infty} \rvert$ into its martingale and finite variation parts then, and denoting by
    \begin{equation*}
        m_{k,t}^{x,y}(\eta,g)
        =
        \lvert \mu_k(x,\eta) - \mu_k(y,\eta) \rvert
        +
         \lvert g(x) - g(y) \rvert
       \sum_{l=1}^K
       \int_{- \infty}^t
       \varphi_{kl}(t-s,\eta) \dif N^{x,\infty}_{l,s}
    \end{equation*}
    the first two terms of~\eqref{equ:modulus}, one has at $(\eta^*,g^*)$ that, for any $k \in 1\hdots K$, and $0\leq t \leq T$,
    \begin{align*}
        \lvert 
        \lambda^{x,\infty}_{k,t}(\eta^*,g^*)
        -
        \lambda^{y,\infty}_{k,t}(\eta^*,g^*) \rvert
        &\leq 
         m_t^{x,y}(\eta^*,g^*)
        +
        \sum_{l=1}^K
        \int_0^t \bar{\varphi}_{kl}^*(t-s)
        \dif \tilde{M}_{l,s}^{x,y}
        \\
        &+
        \sum_{l=1}^K
        \int_0^t
        \bar{\varphi}_{kl}^*(t-s)
        \lvert 
        \lambda^{x,\infty}_{l,s}(\eta^*,g^*)
        -
        \lambda^{y,\infty}_{l,s}(\eta^*,g^*) \rvert
        \dif s,
    \end{align*}
    where $\bar{\varphi}^*$ is defined in~\eqref{equ:bar_phi}. This is the same renewal equation as in~\eqref{equ:first_appearance_of_renewal} in the proof of Lemma~\ref{Lemma:jaisson_like}, with the baseline $\Phi_k[\mu_k]$ replaced by $m_t^{x,y}(\eta^*,g^*)$ and the martingale $(M^T_s)$ replaced by the martingale $(\tilde{M}_s^{x,y})$. Consequently, by the exact same arguments, we obtain that for any $k=1\hdots K$, any $\eta \in \Gamma$, $g \in C[0,1]$, and $0 \leq t \leq T$, 
    \begin{align}
          \lvert 
        \lambda^{x,\infty}_{k,t}(\eta,g)
        -
        \lambda^{y,\infty}_{k,t}(\eta,g) 
        \rvert
        &\leq 
        m_{k,t}^{x,y}(\eta,g)
        +
        \sum_{l=1}^K
        \int_{- \infty}^t 
        \Psi^{g,\eta}(t-s)
        m_{l,s}^{x,y}(\eta^*,g^*)
        \dif s
        \label{equ:dm_form}
        \\
        &+
        \sum_{l=1}^K
        \int_{- \infty}^t 
        \Psi^{g,\eta}_{kl}(t-s)
        \dif \tilde{M}^{x,y}_{l,s},
        \nonumber
    \end{align}
    where we have introduced the resolvent $\Psi^{g,\eta}=\sup_u g(u) \varphi(\cdot,\eta)  + \sup_u g(u)\varphi(\cdot,\eta) \star \Bar{\Psi}^*$. Taking the supremum, then the expectation, 
    \begin{align*}
       \sum_{k=1}^k
        \mathbb{E}[  \lvert 
        \lambda^{x,\infty}_{k,t} (\eta,g)
        -
        \lambda^{y,\infty}_{k,t}(\eta,g)
        \rvert]
        &\leq 
        \sum_{k=1}^K
        \mathbb{E}[ 
        m_{k,t}^{x,y}(\eta,g) ]
        +
        C_k \mathbb{E}[\lvert m_{k,t}^{x,y}(\eta^*,g^*)\rvert ],
    \end{align*}
    where we have used the fact that the $m_{k,t}^{x,y}$ are stationary and
    $
        C_k =  \sum_{l=1}^K \int_0^{\infty} \Psi^{g,\eta}_{kl}(s) \dif s
    $. By Proposition~\ref{prop:moments}, the preceding bound is itself dominated, up to some multiplicative constant, by
    \begin{equation*}
  \omega(g,\lvert x-y\rvert)
          +
        \omega(g^*,\lvert x-y\rvert)
        +
       2 \sum_{k=1}^K  \omega(\mu_k(\cdot,\eta),\lvert x-y\rvert),
    \end{equation*}
    closing the case $p=1$. At $p=2$ then, for any $\vartheta=(\eta,\varpi) \in \Theta$
    \begin{align*}
        \mathbb{E}
        \bigr[ 
        \lvert 
            \lambda^{x,\infty}_{k,t}(\eta,g)
            -
            \lambda^{y,\infty}_{k,t}(\eta,g)
        \rvert^2
        \bigl]
        &\leq         
         3\mathbb{E}\bigl[
         \bigl(
         m_{k,0}^{x,y}(\eta,g)
         \bigr)^2
         \bigr]
         \\
        &+
         3\sum_{l=1}^K
         \Big( 
         \int_{0}^{\infty} \Psi^{g,\eta}_{kl}(s) \dif s \Big)^2
         \mathbb{E}[(m_{l,0}^{x,y}(\eta^*,g^*))^2]
         \\
         &
         +
         3\sum_{l=1}^K
         \Big( 
         \int_{0}^{\infty} \Psi^{g,\eta}_{kl}(s)^2 \mathbb{E}[ 
        \lvert 
            \lambda^{x,\infty}_{k,s}(\eta^*,g^*)
            -
            \lambda^{y,\infty}_{k,s}(\eta^*,g^*)
        \rvert]
        \dif s \Big),
    \end{align*}
    where we have relied on elementary convexity inequalities and $(\tilde{M}_{k,t}^{x,y})$ having a diagonal predictable covariation $[\tilde{M}^{x,y},\tilde{M}^{x,y}]_t$ with coefficients $\int_0^t \lvert \lambda_{k,s}^{x,\infty} (\eta^*,g^*) -  \lambda_{k,s}^{y,\infty} (\eta^*,g^*)\rvert \dif s$. The quadratic case thus deduces from the case $p=1$ together with Proposition~\ref{prop:moments}.
\end{proof}

\begin{Lemma}\label{lemma:uniform_ergodicity} Under Assumptions~\ref{ass:lipschitz} to~\ref{ass:integrability}, for any $C_{\uparrow}(\R^K)$, any $\eta \in \Gamma$ and $g \in C[0,1]$,
    \begin{equation*}
        \sup_{x \in [0,1]}
        \mathbb{E}\Big[ 
        \big\lVert 
            \frac{1}{T}
        \int_0^T 
        \psi( \lambda^{x,\infty}_s(\eta,g))
        \dif s
        -
        \frac{1}{T}
        \int_0^T
        \mathbb{E}\big[ 
          \psi( \lambda^{x,\infty}_0(\eta,g))
        \big] 
        \dif s
        \big\rVert
        \Big]
        \to 0
    \end{equation*}
    in probability as $T \to \infty$.
\end{Lemma}
\begin{proof} Let $\psi \in C_{\uparrow}(\R^K)$, $\eta \in \Gamma$, and $g \in C[0,1]$. We have already obtained the stated convergence pointwise in $x$ in Lemma~\ref{Lemma:stationary_ergodicity}.  It suffices that $\psi(\lambda^{x,\infty}_0(\eta,g))$ be uniformly equi-continuous in $L^1(\Prob)$-norm for the uniform convergence to hold too.  We proceed as in~\cite{ClinetYoshida}, and expand $\psi$, so that
 \begin{equation}\label{equ:expansion}
        \psi(\lambda^{y,\infty}_t(\eta,g) )
        -
        \psi( \lambda^{x,\infty}_t(\eta,g))
        =
        \int_0^1 
        \nabla
        \psi\bigl( X^{x,y}_u(\eta,g)
         \bigr) 
         (\lambda^{y,\infty}_t (\eta,g)- \lambda^{x,\infty}_t(\eta,g)
         )
         \dif u,
    \end{equation}
    where $X^{x,y}_u(\eta,g) = (1-u) \lambda_t^{y,\infty}(\eta,g) + u\lambda_t^{x,\infty}(\eta,g)$. By definition of $C_{\uparrow}(\R^K)$, one has some $C,\gamma>0$ such that $\lvert \nabla \psi(x) \rvert \leq C(1+\lvert x \rvert^\gamma)$ whence, by Proposition~\ref{prop:moments}, for any $p \geq 1$,
    \begin{equation*}
        \sup_{x \in [0,1]}\mathbb{E} \bigl[
        \lvert
        \nabla \psi( X_u^{x,y}(\eta,g) )
        \rvert^p 
        \bigr]
        < \infty.
    \end{equation*}
    The Cauchy-Schwarz inequality together with expansion~\eqref{equ:expansion} then yields 
    \begin{equation*}
        \mathbb{E}[ 
        \lVert 
        \psi(\lambda^{x,\infty}_s(\eta,g))  
        -
        \psi(\lambda^{x,\infty}_s(\eta,g))  
        \rVert]
        =
        \mathcal{O}\big( 
        \mathbb{E}[
        \lVert 
        \lambda^{x,\infty}_s ((\eta,g)
        -
        \lambda^{x,\infty}_s((\eta,g)
        \rVert_2^2 ]
        \big).
    \end{equation*}
    The uniform equicontinuity of $\psi(\lambda^{x,\infty}_0)$ hence deduces from that of $x \mapsto \lambda^{x,\infty}_s$ in Lemma~\ref{lemma:continuity_in_dx}.
\end{proof}

\subsection{The \textsc{kcd} scheme}\label{section:kcd}
We may now proceed to the convergence of $(\lambda^T_t)$
\begin{Lemma}
Under Assumptions~\ref{ass:lipschitz} to~\ref{ass:integrability}, for any $\psi \in C_{\uparrow}(\R)$, any $\eta \in \Gamma$ and $g \in C[0,1]$,
    \begin{equation}\label{equ:error_sum}
    \mathbb{E} \Bigl[ 
        \frac{1}{T}
        \sum_{i=1}^{\nicefrac{T}{\Delta_T}}
        \sum_{k=1}^K 
        \int_{(i-1) \Delta_T}^{i\Delta_T  }
        \lvert 
        \psi(\lambda^T_{k,s}(\eta,g))
        -
        \psi(\lambda_{k,s}^{i \Delta_T,\infty}(\eta,g)) 
        \rvert \dif s
        \Bigr]
        \to 0,
    \end{equation}
    as $T \to \infty$.
\end{Lemma}
\begin{proof}
    Let $\psi \in C_{\uparrow}(\R)$ and let $\vartheta \in \Theta$, $g \in C[0,1]$. Let us first remark that the first term in the sum~\eqref{equ:error_sum} is of order $\mathcal{O}(\frac{1}{T})$ by consequence of Proposition~\ref{prop:moments}, hence we may work only with summands wherein $(i-1) \geq 1$. The proof is then similar to that of Lemmata~\ref{Lemma:jaisson_like} and~\ref{lemma:continuity_in_dx}: we intend to bound the difference in terms of a renewal equation. See in particular that, expanding $\psi$ as in~\eqref{equ:expansion} and using Proposition~\ref{prop:moments} still, one has for any $k =1 \hdots K$ and $0 \leq s \leq T$, 
    \begin{equation*}
    \sum_{k=1}^K \mathbb{E}\bigl[ 
        \lvert  \psi(\lambda^T_{k,s}(\eta,g))
        -
        \psi(\lambda_{k,s}^{i \Delta_T,\infty}(\eta,g))
        \rvert\bigr]
        \leq 
        C_{\psi}
        \sum_{k=1}^K 
        \mathbb{E}\bigl[ (\lambda^T_{k,s}((\eta,g))
        -
        \lambda_{k,s}^{i \Delta_T,\infty}(\eta,g))^2 
        \bigr],
    \end{equation*}
    where $C_{\psi}$ is a constant depending  on $\psi$. It then suffices to show that the uniform approximation error $\sup_{\{0 \leq x<s<x+\Delta_T \leq 1\}} \mathbb{E}[ (\lambda^T_{k,s}(\eta,g))
        -
        \lambda_{k,s}^{x,\infty}(\eta,g))^2 ]$ is at most of order $\frac{\Delta_T}{T}$. Then, let $T>0$, $j \in \mathbb{N}^\star$, $t \in [(j-1)\Delta_T,j\Delta_T]$. For any $k =1 \hdots K$,
        \begin{align}
            \lvert \lambda^T_{k,t}
            (\eta,g)
            -
            \lambda^{j\Delta_T ,\infty}_{k,t}
            (\eta,g)
            \rvert
            &\leq 
            \lvert 
            \mu_k\bigl(\frac{t}{T},\eta\bigr)
            -
            \mu_k\bigl(j\frac{\Delta_T}{T},\eta\bigr)
            \rvert 
            \nonumber
            \\
            &+
            \sup_{x \in [0,1]} g(x)
            \sum_{l=1}^K
            \int_{-\infty}^{0}\varphi_{kl}(t-s,\eta) 
            \dif N^{j \Delta_T}_{l,s}
            \label{equ:yet_another_renewal}
            \\
            &+
            \lvert 
                g\bigl(\frac{t}{T}\bigr)
                -
                g\bigr(j \frac{\Delta_T}{T} \bigl) 
            \rvert
            \sum_{l=1}^K
            \int_0^T
            \varphi_{kl}(t-s,\eta) 
            \dif N^{T,\infty}_{l,s}
            \nonumber
            \\
            &+
            \sup_{x}g(x)
            \sum_{l=1}^K
            \int_0^T  
            \varphi_{kl}(t-s, \eta)
            \dif \lvert 
            N^T_{l,s}
            -
            N^{j \Delta_T,\infty}_{l,s}
            \rvert
            ,
            \nonumber
        \end{align}
        where, for any $x \in [0,1]$, the random measure $ \dif \lvert 
            N^T_{l,s}
            -
            N^{x,\infty}_{l,s}
            \rvert$ counts event-times occurring in either one of $N^{T}$ or $N^{x,\infty}$ but not the other. Denote by $\textswab{m}_t(\eta,g)$ the sum of the first three terms on the right-hand side of~\eqref{equ:yet_another_renewal} and introduce 
            the martingale $(\bar{M}^{j,T}_s)$ defined by
            \begin{equation*}
                \Bar{M}^{j,T}_{k,s} = 
                \int_0^t
                \dif \lvert
                N_{k,s}^{T}
                -
                N_{k,s}^{j \Delta_T,\infty}\rvert 
                -
                \lvert
                \lambda_{k,s}^{T}(\eta^*,g^*)
                -
                \lambda_{k,s}^{j \Delta_T,\infty}(\eta^*,g^*)\rvert
                \dif s.
            \end{equation*} 
            Observe that~\eqref{equ:yet_another_renewal} coincides with inequality~\eqref{equ:modulus} in the proof of Lemma~\ref{lemma:continuity_in_dx}, with the random measure $\dif \lvert N^{x,\infty} - N^{y,\infty} \rvert$ replaced by $\dif \lvert N^{T} - N^{j \Delta_T,\infty} \rvert$, and $m^{x,y}_t(\eta,g)$ replaced by $\textswab{m}_t(\eta,g)$. By the same arguments as for Lemma~\ref{lemma:continuity_in_dx} then, for any $\vartheta \in \Theta$ and $g \in C[0,1]$, 
            \begin{equation*}
                  \lvert \lambda^T_{k,t}((\eta,g))
                -
                \lambda^{x,\infty}_{k,t}((\eta,g))
                \rvert 
                \leq 
                \textswab{m}^{x,y}_t(\vartheta)
                +
                \int_0^t \Psi^{\eta,g}(t-s) 
                 \textswab{m}_s(\eta^*,g^*)
                 \dif s
                 +
                 \int_0^t \Psi^{\eta,g}(t-s) 
                 \textswab{m}_s(\eta^*,g^*)
                 \dif \tilde{M}^{x,T}_s.
            \end{equation*}
            which is indeed the same decomposition as in~\eqref{equ:jaisson_form} or~\eqref{equ:dm_form}. In view of the proof of Lemma~\ref{lemma:continuity_in_dx}, the desired continuity will then follow from showing that $\mathbb{E}[\textswab{m}_{k,t}(\eta,g)^2] \to 0$ uniformly in $t$ as $T \to \infty$ for any $k = 1 \hdots K$ and any $\vartheta \in \Theta$. Recall then that $\Delta_T= o(T)$. We work over sub-intervals wherein $t \in [(j-1) \Delta_T,j\Delta_T]$, hence
            \begin{equation*}
                \sum_{k=1}^K \lvert \mu_k(\frac{t}{T},\eta) - \mu_k(j \frac{\Delta_T}{T},\eta) \rvert^2 = \mathcal{O}( \omega\bigl(\sum \mu_k(\cdot,\eta), \Delta_T)^2\bigr)
            \end{equation*}
            which vanishes as $T \to \infty$ uniformly in $t$ for any $\eta \in \Gamma$ as a consequence of the uniform continuity of the $\mu_k(\cdot,\eta)$ over $[0,1]$. Similarly, for any $\eta \in \Gamma$, $g \in C[0,1]$, any $0 \leq t \leq T$ and $l = 1 \hdots K$,
            \begin{equation*}
                \bigl\lvert 
                g\big(  \frac{t}{T} \bigr)
                -
                g\big(  j \frac{\Delta_T}{T} \bigr)
                \bigr\rvert^2
                \mathbb{E}\Bigl[ 
                \bigl(
                \int_0^t
                \varphi_{kl}(t-s,\eta) \dif N^T_{l,s}
                \bigr)^2
                \Bigr]
                =
                \mathcal{O}( \omega(g,\frac{\Delta_T}{T}))
            \end{equation*}
        thanks to Proposition~\ref{prop:moments}, and one concludes again from the uniform continuity of $g$. There remains only the third term on the right hand side of~\eqref{equ:yet_another_renewal} which is bounded, up to some multiplicative constant, by
        \begin{align}\label{equ:ready_for_doob}
            \sum_{l=1}^K
            \mathbb{E}\Bigl[ 
            \Big(
            \int_{-\infty}^{0}
            \varphi_{kl}(t-s,\eta) 
            \dif N^{x,\infty}_{l,s}\Big)^2
            \Bigr]
            & =
            \sum_{l=1}^K
            \mathbb{E}\Bigl[ 
            \Bigl(
            \int_{-\infty}^{0}
            \varphi_{kl}(t-s,\eta) 
            (\dif M^{x,\infty}_{l,s} 
            +
            \lambda_{l,s}^{x,\infty}(\eta^*,g^*) \dif s)
            \Bigr)^2
            \Bigr].
        \end{align}
        Then, using that $(\int_{-\infty}^s \varphi(t-u,\eta) \dif \bar{M}^{x,T}_u)$ is a martingale with covariation $(\int_0^t \varphi_{kl}(t-s,\eta) \dif N_{l,s}^{x,\infty})$ for any $t \in \R_+$, $x \in [0,1]$, and $k= 1\hdots K$, one has for any $0 \leq t \leq T$ and $k=1 \hdots K$, 
        \begin{equation*}
             \sum_{l=1}^K
            \mathbb{E}\Bigl[ 
            \Big(
            \int_{-\infty}^{0}
            \varphi_{kl}(t-s,\eta) 
            \dif M^{x,\infty}_{l,s}\Big)^2
            \Bigr]
            \leq 
            \Bigl( \sum_{l=1}^K \int_t^{\infty} \varphi_{kl}^2(s,\eta) \dif s 
            \Big)
            \sup_{x \in [0,1]} \mathbb{E}[\lambda_{l,0}^{x,\infty}(\eta^*,g^*)],
        \end{equation*}
        and, by application of Jensen's inequality with the normalisation $\int_t^{\infty} \varphi_{kl} (s,\eta)\dif s$,
        \begin{equation*}
           \sum_{l=1}^K 
           \mathbb{E}\Bigl[ \Bigl( \int_0^t \varphi_{kl}(t-s,\eta) \lambda_{l,s}(\eta^*,g^*) \Bigr)^2 \Bigr]
           \leq
           \sum_{l=1}^K
           \Bigl( \int_t^{\infty} \varphi(s,\eta) \dif s
           \Bigr)^2
           \mathbb{E}\big[ (\lambda^{x,\infty}_{l,0} (\eta^*,g^*))^2\big].
        \end{equation*}
        Recall that we have discarded the first summand wherein $t \in [0,\Delta_T)$, hence work with $t \geq \Delta_T$. Re-inserting the last two inequalities into the bound~\eqref{equ:ready_for_doob} and applying Proposition~\ref{prop:moments}, we have therefore established 
        \begin{equation*}
            \mathbb{E} 
             \Bigl[ 
             \bigl( 
             \sup_x g(x)
            \sum_{l=1}^K
            \int_{-\infty}^{0}\varphi_{kl}(t-s,\eta^*) 
            \dif N^{j \Delta_T}_{l,s}
            \bigr)^2 \Bigr]
            =
            \mathcal{O}\Bigl(
            \sum_{l=1}^K
            \bigl(\int_{\Delta_T}^{\infty} \varphi_{kl}(s,\eta) \dif s\bigr)^2
            +
            \int_{\Delta_T}^{\infty} \varphi^2_{kl}(s,\eta)  \dif s
            \Bigr),
        \end{equation*}
        thereby ending the proof in view of the integrability properties of the $\varphi(\cdot,\eta)$ under Assumption~\ref{ass:integrability}.
        \end{proof}
    \begin{Lemma}\label{lemma:penultimate_approximation}
    Under Assumptions~\ref{ass:lipschitz} to~\ref{ass:integrability}, for any $\eta \in \Gamma$ and $g \in C[0,1]$,
    \begin{equation*} 
        \bigl\lVert \frac{1}{T}
        \sum_{i=1}^{\nicefrac{T}{\Delta_T}}
        \int_{(i-1) \Delta_T}^{i\Delta_T }
        \psi(\lambda^{(i-1) \Delta_T,\infty}_t(\eta,g))
        \dif t
        -
        \int_{(i-1) \Delta_T}^{i\Delta_T }
        \mathbb{E} \big[ \psi(\lambda_0^{(i-1) \Delta_T,\infty}(\eta,g) )
        \big] 
        \dif s
        \bigr\rVert 
        \to 
        0
    \end{equation*}
    as $T \to \infty$ in probability under $\Prob(\eta^*,g^*)$.
    \end{Lemma}
    \begin{proof} By the triangle inequality, the expectation of the expression of interest is bounded by
        \begin{equation*}
        \sup_{x \in [0,1]}
         \mathbb{E}\Bigl[
         \frac{1}{\Delta_T}
         \Bigl\lvert
        \int_{0}^{ \Delta_T}
        \psi(\lambda^{x \Delta_T ,\infty}_t(\eta,g))
        \dif t
        -
        \mathbb{E} \big[ \psi(\lambda_0^{x,\infty} (\eta,g))
        \big] 
        \Bigr\rvert
        \Bigr],
        \end{equation*}
        which indeed goes to $0$ as $T \to \infty$ via an immediate adaptation of Lemma~\ref{lemma:uniform_ergodicity}.
    \end{proof}
\begin{Lemma}\label{equ:final_approximation}
Under Assumptions~\ref{ass:lipschitz} to~\ref{ass:integrability},
        \begin{equation*}
            \frac{1}{T}
            \sum_{i=1}^{\nicefrac{T}{\Delta_T}}
            \int_{(i-1)\Delta_T}^{i\Delta_T}
            \mathbb{E}[\psi(\lambda^{(i-1) \Delta_T,\infty}_{l,0})] \dif s
            \to 
            \int_0^1  \mathbb{E}[\psi(\lambda^{x,\infty}_{l,0})]  \dif x,
        \end{equation*}
     as $T \to \infty$ in probability under $\Prob(\eta^*,g^*)$.
    \end{Lemma}
    \begin{proof}
        It suffices to remark, as in~\cite{kwan2023asymptotic}, that 
        \begin{equation*}
             \frac{1}{T}
            \sum_{i=1}^{\nicefrac{T}{\Delta_T}}
            \int_{(i-1)\Delta_T}^{(i-1)\Delta_T+\Delta_T}
            \mathbb{E}[\psi(\lambda^{(i-1) \Delta_T,\infty}_{l,0})] \dif s
            =
            \sum_{i=1}^{\nicefrac{T}{\Delta_T}}
            \frac{\Delta_T}{T}
            \mathbb{E}[\psi(\lambda^{(i-1) \Delta_T,\infty}_{l,0})]
        \end{equation*}
        is a Riemann sum over $[0,1]$, whence it converges towards the desired limit.
    \end{proof}

    The arguments we have presented in Remarks~\ref{remark:extend_to_derivatives} and~\ref{remark:extend_again_to_derivatives} apply to Lemmata~\ref{lemma:penultimate_approximation} and~\ref{equ:final_approximation}. Thus they extend from the $(\lambda^{x,\infty}_t(\eta,g))$ to derivatives of $(\lambda_t^{x,\infty}(\vartheta))=(\lambda_t^{x,\infty}(\eta,g(\cdot,\varpi)))$ according to $\vartheta=(\eta,\varpi)$ in the same manner Lemma~\ref{Lemma:stationary_ergodicity} extends into Lemma~\ref{lemma:stationary_ergodicity_derivative}. Gathering our three successive approximations, one finds that, for any continuous $\psi \in \C_{\uparrow}(E)$, 
    \begin{equation*}
        \frac{1}{T} \int_0^T \psi(\lambda^T_t(\eta^*,g^*),\lambda^T_t(\vartheta),
        \partial_{\vartheta} \lambda^T_t(\vartheta), \partial_{\vartheta}^{\otimes 2} \lambda^T_t(\vartheta)) \dif t 
    \end{equation*}
    converges in probability under $\Prob(\eta^*,g^*)$ towards 
    \begin{equation*}
         \int_0^1 
         \mathbb{E}[\psi(\lambda^{x,\infty}_0(\eta^*,g^*),\lambda^{x,\infty}_0(\vartheta),
        \partial_{\vartheta} \lambda^{x,\infty}_0(\vartheta), \partial_{\vartheta}^{\otimes 2} \lambda^{x,\infty}_0(\vartheta)) ]\dif x,
    \end{equation*}
    which is Proposition~\ref{prop:ergodicity} at $u=1$, pointwise in $\vartheta$. Extension to any $u \in [0,1]$ proceeds from replacing the approximation grid $[(j-1)\Delta_T,j\Delta_T)$ in favour of the grid $[(j-1)\Delta_Tu,j\Delta_Tu)$ with step $u \Delta_T$. The final limit in Lemma~\eqref{equ:final_approximation} then takes the form $\int_0^1 u \mathbb{E}[\psi(\lambda^{ux,\infty}_0)] \dif x$, and one concludes with the change of variable $x \mapsto ux$. Finally, the uniformity in $\vartheta \in \Theta$ is a consequence of Proposition~\ref{prop:moments} (see~\cite[Proposition 3.8]{ClinetYoshida}). 

\section{Proof of Proposition~\ref{Proposition:robust}}\label{section:proof_consistent}

We proceed to the proof of Proposition~\ref{Proposition:robust}, which is set in the semi-nonparametric context.  The relevant \textsc{mle}s are denoted by
\begin{align*}
(\hat{\eta}^{\mathfrak{d}}_T,\hat{g}^\mathfrak{d}_T) 
   &=
   \argmax_{ \eta  \in \Gamma, \hspace{0.1cm} g \in \B_{\mathfrak{d}}[0,1] } \mathcal{L}_T( \eta, g), \\
   (\hat{\eta}^{0,\mathfrak{d}}_T, \hat{g}^{0,\mathfrak{d}}_T) 
   &= 
   \argmax_{ \eta \in \Gamma, \hspace{0.1cm} C >0} \mathcal{L}_T( \eta, C).
\end{align*}

We follow the strategy described in Section~\ref{section:nonparametric}, and seek some estimate of the discrepancy within the normalized likelihood ratio $\frac{1}{T}\Lambda^\mathfrak{d}_T$, as quantified in Lemma~\ref{Lemma2} below.

    \begin{Lemma}\label{Lemma2} Under Assumption~\ref{ass:lipschitz} to~\ref{ass:integrability}, for any $\epsilon>0$, there is some $\mathfrak{d}>0$ such that
\begin{equation*}
    \sup_{\vartheta \in \Theta, \hspace{0.1cm} g \in \mathbb{B}_{\mathfrak{d}}[0,1]  } \mathcal{L}_T(\vartheta,{g_\mathfrak{d}})
    \geq 
    \mathcal{L}_T(\vartheta^*,g^*)
    -T \varepsilon
    +
    o_{\Prob(\vartheta^*,g^*)}(T)
\end{equation*}
as $T \to \infty$.
\end{Lemma}
\begin{proof}
    For any $D \in \mathbb{N}^\star$, let us introduce the pseudo-estimator $(\eta^*,g^*_D)$ with $g^*_D$ defined by
    \begin{equation*}
        g^*_D(t)= \sum_{k=0}^D \binom{D}{k} t^k (1-t)^{D-K} g^{*}\big( \frac{k}{D}\big)
    \end{equation*}
    so that $\lVert g^*_D - g^*\rVert_{L^\infty} \to 0$ as $D \to \infty$. By Proposition~\eqref{prop:ergodicity}, for any $D \in \mathbb{N}^\star$,
    \begin{equation*}
        \frac{1}{T}
        \mathcal{L}_T(\eta^*,g^*)
        -
        \frac{1}{T}
        \mathcal{L}_T(\eta^{*},g^{*}_D)
        =
        \int_0^1 \mathfrak{L}(x,\eta^*,g^*_D) \dif x
        +
        o_{\Prob(\eta^*,g^*)}(1),
    \end{equation*}
    where 
    \begin{equation*}
        \mathfrak{L}(x,\eta^*,g^*_D)
        =
        \sum_{k=1}^K
        \mathbb{E}\Bigl[\lambda^{x,\infty}_{k,0}(\eta^*,g^*)
        \Bigl\{
            \log\Big( 
                \frac{\lambda^{x,\infty}_{k,0}(\eta^*,g^*_D)}{\lambda^{x,\infty}_{k,0}(\eta^*,g^*)}
            \Big)
            -
            \Big( 
                \frac{\lambda^{x,\infty}_{k,0}(\eta,g^*_D)}{\lambda^{x,\infty}_{k,0}(\eta^*,g^*)}
                -
                1
            \Big)
        \Bigr\}\Bigr].
    \end{equation*}
    One may pick $D$ sufficiently large that the preceding display obeys $  -\epsilon \leq \mathfrak{L}(x,\eta,g^*_D)\leq0$, see again the proof of Proposition~\ref{prop:consistent} for the existence of a well-separated maximum of $x \mapsto \log(x)-(x-1)$.  By definition of $(\hat{\eta}_T^{\mathfrak{d}},\hat{g}^\mathfrak{d}_T)$, one has $\mathcal{L}_T(\hat{\eta}_T^{\mathfrak{d}},\hat{g}^\mathfrak{d}_T)>\mathcal{L}_T(\eta^*,g^D)$ as long as $\mathfrak{d}\geq D$.  The Lemma thus follows by setting $D$ large enough, and then letting $T \to \infty$.
\end{proof}

\begin{Lemma}\label{Lemma1}
    Under Assumptions~\ref{ass:lipschitz} to~\ref{ass:convexity}, there is some $\mathcal{K}>0$ such that
    \begin{equation*}
     \mathcal{L}_T(\eta^*,g^{*})
     \geq \sup_{\eta,C} \mathcal{L}(\eta,C)
     +
     T\mathcal{K} 
     +
     o_{\Prob(\eta^*,g^*)}(T).
    \end{equation*}
\end{Lemma}
\begin{proof}
    Working as we have done in the proof of Lemma~\ref{Lemma2} we obtain that
    \begin{equation*}
        \frac{1}{T}\mathcal{L}_T(\eta^*,g^*)
        -
        \frac{1}{T}
        \sup_{\eta \in \Gamma C>0}
        \mathcal{L}_T(\eta,C)
        \geq 
        \inf_{\eta \in \Gamma, C>0} 
        \Bigl( 
        \int_0^1
        -\mathfrak{L}(x,\eta,C)
        \dif x
        \Bigr)
        +
        o_{\Prob(\eta^*,g^*)}(1),
    \end{equation*}
    where, using the elementary inequality
    $
        (x-1)-\ln(x) \geq\frac{1}{2}(x-1)^2(x+1)^{-1}
    $,
    the function
    \begin{equation*}
         -\mathfrak{L}(x,\eta,C)
        =
        \sum_{k=1}^K
        \mathbb{E}\left[ \lambda^{x,\infty}_{k,0}(\eta^*,g^{*})
        \Big\{ \Big( \frac{\lambda^{x,\infty}_{k,0}(\eta,C)}{\lambda^{x,\infty}_{k,0}(\eta^*,g^{*})} - 1\Big) - \log\Big( \frac{\lambda_{k,0}^{x,\infty}(\eta,C)}{\lambda^{x,\infty}_{k,0}(\eta^*,g^*)}\Big)\Big\}\right]
    \end{equation*}
    is non-negative and bounded from below by
    \begin{equation}\label{equ:elementary_bound}
        \sum_{k=1}^K
         \mathbb{E}\left[  \frac{1}{2} \frac{ (\lambda^{x,\infty}_{k,0}(\eta^*,g^*)-\lambda^{x,\infty}_{k,0}(\eta,C) )^2}{ \lambda^{x,\infty}_{k,0}(\eta^*,g^*)+\lambda^{x,\infty}_{k,0}(\eta,C)}\right].
    \end{equation}
   Applying the reverse Hölder inequality at $p=2$ with the measure $\dif \Prob \otimes \dif x$,~\eqref{equ:elementary_bound} is larger than
   \begin{equation}\label{equ:holder_bound}
     \frac{1}{2} 
     \sum_{k=1}^K \frac{ \int_0^1 \mathbb{E}\bigl[ \lvert \lambda_{k,0}^{x,\infty}(\vartheta^*,g^*) -  \lambda_{k,0}^{x,\infty}(\vartheta,C)  \rvert \bigr]^2 \dif x}
     {\int_0^1 \mathbb{E}[ \lambda_{k,0}^{x,\infty}(\vartheta^*,g^*) +  \lambda_{k,0}^{x,\infty}(\vartheta,C)  ]\dif x}.
   \end{equation}
   Using that $\sup_{\eta,C} \mathbb{E}[\lambda^{,x\infty}(\eta,C)]<\infty$, one obtains some constant $C_{\vartheta^*,g^*}>0$ such that
    \begin{equation}\label{equ:quadratic_bound}
         \frac{1}{T}\mathcal{L}_T(\eta^*,g^*)
        -
        \frac{1}{T}
        \sup_{\eta \in \Gamma C>0}
        \mathcal{L}_T(\eta,C)
        \geq 
        C_{\eta^*,g}
        \sum_{k=1}^K
        \int_0^1 \mathbb{E}\bigl[ \lvert 
       \lambda^{x,\infty}_{k,0}(\eta^*,g^*)
        -
        \lambda^{x,\infty}_{k,0}(\eta,C)
        \rvert \bigr]^2
        \dif x.
    \end{equation}
    We have shown in Lemma~\ref{lemma:continuity_in_dx} that $x \mapsto \lambda^{x,\infty}$ is uniformly equicontinuous in $x$ in $L^1(\Prob)$-norm, which bears two useful consequences here.  Firstly, that the previous estimate reaches a minimal value over the compact set $\Theta \times [0,1]$ at some $(\eta_m,C_m)$. Secondly, that it suffices to obtain some $x \in [0,1]$ such that $\lambda^{x,\infty}(\eta^*,g^*) \neq \lambda^{x,\infty}(\eta_m,C_m)$ with positive probability for this minimum to be positive. However, the activation functions $\Phi_k$ are bijective and continuous under Assumption~\ref{ass:convexity}, and the random measure $N^{x,\infty}$ positively charges any interval in $\R$ under Assumption~\ref{ass:no_inhibition}. Hence for any $x \in [0,1]$, that $\lambda^{x,\infty}(\eta^*,g^*) = \lambda^{x,\infty}(\eta_m,C_m)$ holds $\Prob(\eta^*,g^*)$-a.s would imply that
\begin{equation}\label{equ:condition_for_0}
    \varphi(\cdot,\eta^*) = \frac{g^*(x)}{C_m}\varphi(\cdot,\eta_m).
\end{equation}
 Condition~\eqref{equ:condition_for_0} uniquely determines $(\eta_m,C_m)$ under Assumption~\ref{ass:identifiability}, hence it may not hold for two different $x,x' \in [0,1]$ such that $g^*(x) \neq g^*(x')$.  It therefore suffices that $g^*$ be non-constant for the lower bound in~\eqref{equ:quadratic_bound} to be positive. 
\end{proof}
Before we proceed to the rest of the proof of Proposition~\ref{prop:consistent}, let us revisit the bound~\eqref{equ:elementary_bound}. 

\begin{remark}\label{remark:explicit_d}
    Suppose that $\eta^*$ is known and one estimates $g$ only. Applying Jensen's inequality in lieu of Hölder's yields the bound~\eqref{equ:holder_bound} with the absolute value omitted from the expectation. Our final lower bound then takes the form
 \begin{equation*}
     \inf_{\{C>0\}}
     \sum_{k=1}^K 
     \int_0^1
     \mathbb{E}[\lambda_{k,0}^{x,\infty}(\eta^*,g^*)](g^*(x)-C)^2 \dif x
     \geq 
     C_{\Phi,\mu}
     \inf_{\{C>0\}}\sum_{k=1}^K\int_0^1 (g^*(x)-C)^2 \dif x,
 \end{equation*}
 where  $ C_{\Phi,\mu}= \sum_{k=1}^K \inf_{x\in [0,1]}
     \Phi_k[ \mu_k(x,\vartheta^*)]$.
\end{remark}

    \begin{Lemma}
        For any non-constant $g^*$, there is some $M>0$ such that, for any $\mathfrak{d} >M$, $\Lambda_T^{\mathfrak{d}} \to \infty$  under $\Prob(\eta^*,g^*)$  with probability going to $1$ as $T \to \infty$ .
    \end{Lemma}
    \begin{proof}
        For any $\mathfrak{d}\in \mathbb{N}^\star$,
        \begin{align*}
            \sup_{\eta \in \Gamma,g \in \mathbb{R}_{\mathfrak{d}[X]}}\mathcal{L}_T(\eta,g  )
            &- 
            \sup_{\eta \in \Gamma,C>0}\mathcal{L}_T(\eta,C)
            \\
            &=
            \bigl(  \sup_{\eta \in \Gamma,g \in \mathbb{R}_{\mathfrak{d}[X]}}\mathcal{L}_T(\eta,g) 
            -\mathcal{L}_T(\eta^*,g^*)
            \bigr)
            +
            \bigl(
            \mathcal{L}_T(\eta^*,g^*)
            -
              \sup_{\eta \in \Gamma,C>0}\mathcal{L}_T(\eta,C)
            \bigr)
        \end{align*}
        By Lemma~\ref{Lemma1}, the second term of the preceding sum is larger than  $\mathcal{K}T +o_{\Prob(\eta^*,g^*)}(T)$ for some $\mathcal{K}>0$. By Lemma~\ref{Lemma2} the first one may be chosen larger than $-\frac{1}{2}\mathcal{K}T +o_{\Prob(\eta^*,g^*)}(T)$ by picking $\mathfrak{d}$ large enough, indepently from the value of $T$. Consequently, there is some $\mathfrak{d}^*$ such that, for any $\mathfrak{d}\geq \mathfrak{d}^*$,
        \begin{equation*}
            \Lambda^\mathfrak{d}_T =
            2 (
              \sup_{\eta \in \Gamma,g \in \mathbb{R}_{\mathfrak{d}[X]}}\mathcal{L}_T(\eta,g) -    \sup_{\eta \in \Gamma,C>0}\mathcal{L}_T(\eta,C)
            )
            \geq 
           \mathcal{K}T +o_{\Prob}(T),
        \end{equation*}
        which goes to $\infty$ in probability as $T \to \infty$. Recalling furthermore that the error of the Bernstein approximation refines at rate $O(\mathfrak{d}^{-\nicefrac{1}{2}})$, Remark~\ref{remark:explicit_d_in_special_case} deduces from Remark~\ref{remark:explicit_d}.
    \end{proof}

    \paragraph{\textbf{Acknowledgements.}} Thomas Deschatre and Pierre Gruet acknowledge support from the \textit{FiME} Lab.
\bibliographystyle{apalike}
\bibliography{bibli}

\end{document}